\documentclass[reqno, 12pt]{amsart}
\usepackage{amsmath,amssymb,amsfonts}
\usepackage{graphicx}
\usepackage[title]{appendix}
\usepackage{euscript}
\usepackage{enumerate}
\usepackage{verbatim}
\usepackage{epsfig}

\newtheorem{theorem}{Theorem}
\newtheorem{corollary}{Corollary}
\newtheorem{lemma}{Lemma}

\newtheorem{example}{Example}

\renewcommand{\epsilon}{\varepsilon}
\renewcommand{\le}{\leqslant}
\renewcommand{\leq}{\leqslant}
\renewcommand{\ge}{\geqslant}
\renewcommand{\geq}{\geqslant}
\newcommand{\eps}{\varepsilon}

\newcommand{\vphi}{\varphi}
\newcommand{\ds}{\, ds}

\newcommand{\mR}{\mathcal R}

\DeclareMathOperator{\e}{e}

\usepackage{dsfont}
   \newcommand{\N}{\ensuremath{\mathds N}}
   \newcommand{\R}{\ensuremath{\mathds R}}
\usepackage{pstricks}
\begin{document}
\title[]
   {A consistent discrete version of a non-autonomous SIRVS model}
\author{Joaquim P. Mateus}
\address{J. Mateus\\
   Unidade de Investiga\c{c}\~ao para o Desenvolvimento do Interior (UDI)\\
   Instituto Polit\'ecnico da Guarda\\
   6300-559 Guarda\\
   Portugal}
\email{jmateus@ipg.pt}
\author{C\'esar M. Silva}
\address{C. Silva\\
   Departamento de Matem\'atica\\
   Universidade da Beira Interior\\
   6201-001 Covilh\~a\\
   Portugal}
\email{csilva@ubi.pt}
\urladdr{www.mat.ubi.pt/~csilva}
\author{Sandra Vaz}
\address{S. Vaz\\
   Departamento de Matem\'atica\\
   Universidade da Beira Interior\\
   6201-001 Covilh\~a\\
   Portugal}
\email{svaz@ubi.pt}
\date{\today}
\thanks{Joaquim P. Mateus, C\'esar M. Silva and Sandra Vaz were partially supported by FCT through CMA-UBI (project UID/MAT/00212/2013)}
\subjclass[2010]{92D30, 37B55,	65L20, 39A30}
\keywords{Epidemic model, non-autonomous, global stability, numerical method}
\begin{abstract}
A family of discrete non-autonomous SIRVS models with general incidence is obtained from a continuous family of models by applying Mickens non-standard discretization method. Conditions for the permanence and extinction of the disease and the stability of disease-free solutions are determined. The consistency of those discrete models with the corresponding continuous model is discussed: if the time step is sufficiently small, when we have extinction (respectively permanence) for the continuous model we also have extinction (respectively permanence) for the corresponding discrete model. Some numerical simulations are carried out to compare the different possible discretizations of our continuous model using real data.
\end{abstract}
\maketitle
\section{Introduction}

Most of the epidemiological models in the literature are continuous models. In spite of this, recently, there has been a growing interest in discrete-time models~\cite{Diekmann-Heesterbeek-W-2000,Tan-Gao-Fan-DDNS-2015,Wang-Teng-Rehim-DDNS-2014,Wei-Le-DDNS-2015,Liu-Peng-Zhang-AML-2015,Hua-Teng-Zhang-MCS-2014,Hu-Teng-Jia-Zhang-Zhang-ADE-2014}.
In this work, we will use Mickens nonstandard difference (NSFD) scheme to achieve a discretization of a family of continuous epidemiological models with vaccination and general incidence function considered in~\cite{Pereira-Silva-Silva}. We have multiple objectives: firstly, we want to obtain conditions for extinction and permanence of the disease for the discrete family; next, having a continuous and a corresponding discrete family of models, we wish to discuss the problem of consistency of the discrete models with the corresponding continuous ones; finally, we intend to present some simulation results.

The dynamical consistency of a numerical scheme with the associated continuous system is not a precise definition. By the expression ''dynamical consistency'' it is meant that the numerical solutions replicate some of the properties of the continuous systems solutions. For us, dynamical consistency means: whenever there is extinction (respectively permanence) of the disease for the continuous-time model the same holds for the discrete-time one.  Several papers~\cite{Cui-Yang-Zhang-JDEA-2014,Cui-Zhang-JDEA-2015,Ding-Ding-JDEA-2013,Ding-Ding-JDEA-2014,Jang-Elaydi-CAMQ-2003,Mickens-JCAM-1999,Mickens-Washington-JDEA-2012} discuss the dynamical consistency with respect to some particular properties of discrete epidemiological models obtained from continuous models by some NSFD scheme~\cite{Mickens-JDEA-2002}. We note that while the papers cited above consider autonomous models, in the present work we discuss dynamical consistency for a non-autonomous model. To the best of our knowledge, this is the first work where the consistency of a discretized epidemiological model with the original continuous model is discussed in the non-autonomous context.

Regarding our simulation results, we considered two different types of computational experiments. Our first set of simulation results are designed to compare different possible discretizations of our continuous models. After this discussion, we apply our model to a real situation, considering data from the incidence of measles in France in the period 2012-2016. To the possible extent, this data is used to estimate our model parameters and the computational results obtained are compared with real data.

The law of mass action, states that the rate of change in the disease incidence is directly proportional to the product of the number of susceptible and infective individuals, and was the paradigm in the classic models in epidemiology. This is why classical models usually consider a bilinear incidence rate $\beta SI$, where $S$ and $I$ denote respectively the number of susceptible and infective individuals, to model the disease transmission. In spite of this, it is sometimes important to consider other forms of incidence functions. Another usual assumption is the time independence of the parameters model parameters: in fact, the majority of the epidemiological models in the literature are given by a system of autonomous differential or difference equations. Nevertheless, the assumption that the parameters are independent of time is not very realistic in many situations and it is useful to consider non-autonomous models that, for instance, allow the discussion of environmental and demographic effects that change with time~\cite{Khasnis-Nettleman-AMR-2005, Kloeden-Potzsche-Springer-2013}. In this work the family of models considered is non-autonomous and the incidence rates are taken from a large class of functions.

Our model generalizes one obtained by Mickens nonstandard finite difference method from the continuous model~\cite{Pereira-Silva-Silva} (see section~\ref{section:discretization}). In~\cite{Zhang-AMC-2015}, a discrete non-autonomous epidemic model with vaccination and mass action incidence was obtained by Mickens method. We emphasize that, in the particular mass-action case, our model is not exactly similar to the model in~\cite{Zhang-AMC-2015}, although Mickens rules were considered in both. We briefly compare computationally these two slightly different models in Section~\ref{section:simulation}.

The model we will consider is the following
\small{
\begin{equation}\label{eq:ProblemaPrincipal}
\begin{cases}
S_{n+1}-S_n=\Lambda_n-\beta_n\vphi(S_{n+1},I_n)-(\mu_n+p_n)S_{n+1}+\eta_nV_{n+1} \\ I_{n+1}-I_n=\beta_n\vphi(S_{n+1},I_n)+\sigma_n\psi(V_{n+1},I_n)-(\mu_n+\alpha_n+\gamma_n)I_{n+1} \\
R_{n+1}-R_n=\gamma_nI_{n+1} -\mu_nR_{n+1} \\
V_{n+1}-V_n=p_nS_{n+1}-(\mu_n+\eta_n)V_{n+1}-\sigma_n\psi(V_{n+1},I_n)
\end{cases},
\end{equation}
}
$n \in \N$, where the classes $S$, $I$, $R$, and $V$ correspond, respectively, to susceptible, infective, recovered and vaccinated individuals and the parameter functions have the following meanings: $\Lambda_n$ denotes the inflow of newborns in the Susceptible class; the function $\beta_n \vphi$ is the incidence (into the Infective class) from the susceptible individuals; the function $\sigma_n \psi$ is the incidence (into the Infective class) from the vaccinated individuals; $\mu_n$ are the natural deaths; $p_n$ represents the  susceptibles vaccination; $\eta_n$ represents the immunity loss and consequence influx in the susceptible class; $\alpha_n$ are the deaths occurring in the infective class; $\gamma_n$ is the recovery.
We will assume that $(\Lambda_n), (\mu_n), (p_n), (\eta_n), (\alpha_n)$, $(\beta_n)$, $(\sigma_n)$ and $(\gamma_n)$ are bounded and nonnegative sequences and that there are positive constants
$w_\mu$, $w_\Lambda$, $w_p$, $k_{\vphi}$ and $k_{\psi}$ such that:
\begin{enumerate}[H1)]
\item \label{cond-H1} the functions $\vphi:\R^2 \to \R$ and $\psi:\R^2 \to \R$ are nonnegative and differentiable in $(\R_0^+)^2$ and the functions $\R^+_0 \ni x \rightarrow \partial_2 \vphi(x,0)$ and $ \R^+_0 \ni x \rightarrow \partial_2 \psi(x,0)$ are non-decreasing and Lipschitz, with Lipschitz constants $k_{\vphi}$ and $k_{\psi}$.
\item \label{cond-H2} we have $\vphi(x,0)=\psi(x,0)=\vphi(0,y)=\psi(0,y)=0$ for all $x, y \in \R_0^{+}$.
\item \label{cond-P1}
$\displaystyle \limsup_{n \to +\infty} \prod_{k=n}^{n+\omega_\mu} \dfrac{1}{1+\mu_k} <1$;
\item \label{cond-P2} $\displaystyle \liminf_{n \to +\infty} \sum_{k=n+1}^{n+\omega_\Lambda} \Lambda_k > 0  \quad \text{and} \quad \liminf_{n \to +\infty} \sum_{k=n+1}^{n+\omega_p} p_k > 0$;
\item \label{cond-P3}Functions $\R^+ \ni y \mapsto \phi(x,y)/y$ and $\R^+ \ni y \mapsto \psi(x,y)/y$ are non-increasing.
\end{enumerate}

In this work, we prove, when our conditions prescribe extinction (respectively permanence) for the continuous model we also have extinction (respectively permanence) for the corresponding discrete model as long as the time step is smaller than some constant (that depends on some model parameters and on the threshold condition). We also consider a family of examples of the periodic system of period $1$ such that the continuous and the discrete time system with time step $h=1/L$ is not consistent, highlighting the importance of knowing that for time steps smaller than some explicit value we have consistency.

\section{Discretization of the continuous model} \label{section:discretization}

We start with a non-autonomous SIRVS model that is slightly less general than the one considered in~\cite{Pereira-Silva-Silva} and generalizes the one in~\cite{Zhang-Teng-Gao-AA-2008}. Namely, we consider the model:
\begin{equation}\label{eq:ProblemaPrincipal-cont}
\begin{cases}
S'=\Lambda(t)-\beta(t)\vphi(S)I-(\mu(t)+p(t))S+\eta(t)V \\
I'=\left[ \beta(t)\vphi(S)+\sigma(t)\psi(V)-\mu(t)-\alpha(t)-\gamma(t) \right]I \\
R'=\gamma(t)I -\mu(t)R \\
V'=p(t)S-(\mu(t)+\eta(t))V-\sigma(t)\psi(V)I
\end{cases}.
\end{equation}

We assume that the functions $\Lambda, \mu, p, \eta, \alpha$, $\beta$, $\sigma$ and $\gamma$ belong to the class $C^1(\R_0^+)$, are nonnegative and bounded. We also require that:
\begin{enumerate}[C1)]
\item \label{cond-C1} the functions $\vphi:\R \to \R$ and $\psi:\R \to \R$ are nonnegative, non decreasing, differentiable and Lipschitz with Lipschitz constants $k_{\vphi}$ and $k_{\psi}$ respectively;
\item \label{cond-C2} $\vphi(0)=\psi(0)=0$;
\item \label{cond-C3} there is $\omega>0$ such that $\displaystyle \liminf_{t \to +\infty} \int_t^{t+\omega} \mu(s) \ds > 0$.
\end{enumerate}
In order to obtain threshold conditions for model~\eqref{eq:ProblemaPrincipal-cont}, it was considered in~\cite{Pereira-Silva-Silva} the following auxiliary system:
\begin{equation}\label{eq:SistemaAuxiliar-Pereira-Silva-Silva}
\begin{cases}
x'=\Lambda(t) - [\mu(t) + p(t)] x + \eta(t)y \\
y'=p(t)x - [\mu(t)+\eta(t)] y.
\end{cases}
\end{equation}
and for each solution $(x^*(t),y^*(t))$ of~\eqref{eq:SistemaAuxiliar-Pereira-Silva-Silva} with positive initial conditions, it was shown that the numbers
\[
R^\ell_C(\lambda)  = \liminf_{t \to \infty} \int_t^{t+\lambda} \beta(s)\varphi(x^*(s))+\sigma(s)\psi(y^*(s))-\mu(s)-\alpha(s)-\gamma(s) \ds
\]
and
\[
R^u_C(\lambda)  = \limsup_{t \to \infty} \int_t^{t+\lambda} \beta(s)\varphi(x^*(s))+\sigma(s)\psi(y^*(s))-\mu(s)-\alpha(s)-\gamma(s) \ds
\]
are independent of the particular solution.

Using the above numbers, the following results are contained in results obtained in~\cite{Pereira-Silva-Silva}:

\begin{theorem}[Theorem 1 of~\cite{Pereira-Silva-Silva}] \label{teo:Permanence-Pereira-Silva-Silva}
Assume that conditions~C\ref{cond-C1}),~C\ref{cond-C2}) and~C\ref{cond-C3}) hold. Then, if there is a constant $\lambda>0$ such that $R^\ell_C (\lambda) >0$ then the infectives $I$ are permanent
in system~\eqref{eq:ProblemaPrincipal-cont}.
\end{theorem}

\begin{theorem}[Theorem 2 of~\cite{Pereira-Silva-Silva}] \label{teo:Extinction-Pereira-Silva-Silva}
Assume that conditions~C\ref{cond-C1}),~C\ref{cond-C2}) and~C\ref{cond-C3}) hold. Then
if there is a constant $\lambda>0$ such that $R^u_C(\lambda) <0$ then the infectives $I$ go to extinction in system~\eqref{eq:ProblemaPrincipal-cont}.
\end{theorem}

In the literature, several models were discretized using Mickens NSFD schemes~\cite{Allen-Jones-Martin-MB-1994, Allen-MB-1975, Cooke-TPB-1975,DeJong-Diekmann-Heesterbeek-MB-1994, Hattaf-Lashari-Boukari-Yousfi-DEDS-2015, Kaitala-Heino-Getz-BMB-1997, Lena-Serio-MB-1982, Lefevre-Malice-MM-1986, Lefevre-Picard-MB-1989, Longini-MB-1986, Mickens-B-1982, Spicer-BMB-1979, Viaud-MB-1993, West-Thompson-MB-1997}. Next, we will apply Micken's non-standard method to obtain a discrete version of system~\eqref{eq:ProblemaPrincipal-cont}.

Let $\phi:\R_0^+ \to \R$ be a positive continuous function such that
\begin{equation}\label{eq:derivative-denominator}
\lim_{h \to 0} \phi(h)=0.
\end{equation}
Given $h \in \R^+$, we let $t=nh$, with $n \in \N$, and identify $S'(t)$ with
$$
\dfrac{S(nh+h)-S(nh)}{\phi(h)}.
$$
After deciding a non-local representation for the incidence function and that terms that do not correspond to an interaction will be considered in the $n+1$ time, the first equation in~\eqref{eq:ProblemaPrincipal-cont} becomes
\[
\begin{split}
S((n+1)h)-S(nh)
& =\phi(h)\left[\Lambda(nh)-\beta(nh)\vphi(S((n+1)h))I(nh)\right.\\
& \quad \left. -(\mu(nh)+p(nh))S((n+1)h)+\eta(nh)V(nh+h)\right].
\end{split}
\]
Writing $S_n=S(nh)$, $I_n=I(nh)$, $V_n=V(nh)$, $\Lambda_n=\phi(h)\Lambda(nh)$, $\beta_n=\phi(h)\beta(nh)$,
$\mu_n=\phi(h)\mu(nh)$, $p_n=\phi(h)p(nh)$ and $\eta_n=\phi(h)\eta(nh)$, we have
 $$S_{n+1}-S_n=\Lambda_n-\beta_n\vphi(S_{n+1})I_n-(\mu_n+p_n)S_{n+1}+\eta_nV_{n+1}.$$
Proceeding similarly for the other equations, we obtain the following discrete model
\small{
\begin{equation}\label{eq:ProblemaPrincipal-disc}
\begin{cases}
S_{n+1}-S_n=\Lambda_n-\beta_n\vphi(S_{n+1})I_n-(\mu_n+p_n)S_{n+1}+\eta_nV_{n+1} \\ I_{n+1}-I_n=\beta_n\vphi(S_{n+1})I_n+\sigma_n\psi(V_{n+1})I_n-(\mu_n+\alpha_n+\gamma_n)I_{n+1} \\
R_{n+1}-R_n=\gamma_nI_{n+1} -\mu_nR_{n+1} \\
V_{n+1}-V_n=p_nS_{n+1}-(\mu_n+\eta_n)V_{n+1}-\sigma_n\psi(V_{n+1})I_n
\end{cases},
\end{equation}
}
$n \in \N_0$. We will consider a model that contains this one to obtain some of our results. Namely, based on model~\eqref{eq:ProblemaPrincipal-disc}, in sections~\ref{section:PE} and~\ref{section:SR} we will study model~\eqref{eq:ProblemaPrincipal} that has a more general form for the incidence function.

Now, we need to make some definitions. We say that:
\begin{enumerate}[i)]
\item the infectives $(I_n)$ are \emph{permanent} if for any solution $(S_n,I_n,R_n,V_n)$ of~\eqref{eq:ProblemaPrincipal} with initial conditions $S_0,I_0,R_0,V_0>0$ there are constants $0<m<M$ such that
\[
m < \liminf_{n \to \infty} I_n \le \limsup_{n\to \infty} I_n < M;
\]
\item the infectives $(I_n)$ go to \emph{extinction} if for any solution $(S_n,I_n,R_n,V_n)$ of~\eqref{eq:ProblemaPrincipal} with initial conditions $S_0,I_0,R_0,V_0\ge 0$ we have
    $\displaystyle \lim_{n \to \infty} I_n = 0$.
\end{enumerate}
Similar definitions can be made for the other compartments. For instance, if
there exists constants $0<m<M$ such that for any solution $(S_n,I_n,R_n,V_n)$ of~\eqref{eq:ProblemaPrincipal}
with initial conditions $S_0,I_0,R_0,V_0 > 0$ we have
\[
m < \liminf_{n \to \infty} S_n \le \limsup_{n \to \infty} S_n < M
\]
we say that the susceptibles are permanent.

\section{Permanence and Extinction in the Discrete Model} \label{section:PE}

In this section, we will extend the results obtained for the model with the usual mass action incidence in~\cite{Zhang-AMC-2015} to our generalized family of models. Namely, suitable thresholds are defined and conditions for persistence and extinction of the disease are obtained. As a corollary of our results, we consider the periodic case where we have a unique number that establishes the boundary between the regions of permanence and extinction. Although the proofs of our results are inspired in~\cite{Zhang-AMC-2015}, some difficulties must be dealt with. In particular, it was necessary to understand the right conditions to impose to the incidence functions in order to overcome the technical difficulties.

To lighten the reading, the proofs of our results are presented in appendix~\ref{Appendix:A}.

\subsection{Auxiliary results} \label{subsection:AR}
Consider the auxiliary system,
\begin{equation}\label{eq:SistemaAuxiliar}
\begin{cases}
x_{n+1}=\dfrac{\Lambda_n + \eta_n y_{n+1} +x_n}{1+\mu_n+p_n} \\[2mm]
y_{n+1}=\dfrac{p_n x_{n+1} +y_n}{1+\mu_n+\eta_n}
\end{cases}.
\end{equation}
Note that the auxiliary system describes the behaviour of the system in the absence of infection. If $(\Lambda_n)$, $(\mu_n)$, $(p_n)$, $(\eta_n)$, $(\alpha_n)$, $(\mu_n)$, $(\sigma_n)$ and $(\beta_n)$ are constant sequences then~the linear system \eqref{eq:SistemaAuxiliar} becomes autonomous and corresponds to the linearization of the equations for $(S_n)$ and $(V_n)$ in the classical (autonomous) SIRVS model.

In order to proceed we need to recall some notions. A solution $(u_n)$ of some system of difference equations $u_{n+1}=f_n(u_n)$ is said to be \emph{attractive} if for all $n_0 \in \N$ and all $\eps>0$ there is $\sigma(n_0)>0$ and $T(\eps,n_0,u_0) \in \N$ such that if $(\bar u_n)$ is a solution with $\|u_0-\bar u_0\|<\sigma(n_0)$ then $\|u_n-\bar u_n\|<\eps$, for all $n \ge n_0+T(\eps,n_0,u_0)$. Additionally, if some solution is attractive and we can take $T$ to be only dependent on $\eps$, we say that it is \emph{uniformly attractive}.

The following theorem furnishes some simple properties of system~\eqref{eq:SistemaAuxiliar}.

\begin{lemma}[Lemma 2.2 of~\cite{Zhang-AMC-2015}]\label{lema:auxSystem}
Assume that conditions~H\ref{cond-P1}) and~H\ref{cond-P2}) hold. Then
\begin{enumerate}[i)]
\item 
all solutions $(x_n,y_n)$ of system ~\eqref{eq:SistemaAuxiliar} with initial condition $x_0 \ge 0$ and $y_0 \ge 0$ are nonnegative for all
$n \in \N_0$;
\item \label{cond-2-aux} each fixed solution $(x_n,y_n)$ of~\eqref{eq:SistemaAuxiliar}
is bounded and globally uniformly attractive for all $n \in \N_0$;
\item \label{cond-3-aux} if $(x_n,y_n)$ is a solution of~\eqref{eq:SistemaAuxiliar} and $(\tilde x_n, \tilde y_n)$
is a solution of the system
\begin{equation}\label{eq:SistemaAuxiliar2}
\begin{cases}
x_{n+1}=\dfrac{\Lambda_n + \eta_n y_{n+1} +x_n+f_n}{1+\mu_n+p_n} \\[2mm]
y_{n+1}=\dfrac{p_n x_{n+1} +y_n+g_n}{1+\mu_n+\eta_n}
\end{cases}.
\end{equation}
with $(\tilde
x_0, \tilde y_0)=(x_0,y_0)$ then there is a constant $L>0$,
only depending on $\mu_n$, satisfying
\[
\sup_{n \in \N_0} \left\{ |\tilde x_n - x_n| + |\tilde y_n -y_n|\right\} \le L \sup_{n \in \N_0} \left( |f_n|+|g_n| \right);
\]
\item \label{cond-4-aux}
there exists constants $m,M>0$ such that,
for each solution $(x_n,y_n)$ of~\eqref{eq:SistemaAuxiliar}, we have
\[
\begin{split}
& m \le \liminf_{n \to \infty} x_n \le \limsup_{n \to \infty} x_n \le M, \\
& m \le \liminf_{n \to \infty} y_n \le \limsup_{n \to \infty} y_n \le M.
\end{split}
\]
\item \label{cond-5-aux} when the system~\eqref{eq:SistemaAuxiliar} is $\omega$--periodic, it has a unique positive $\omega$--periodic solution which is globally uniformly attractive.
\end{enumerate}
\end{lemma}

We have the following lemma
\begin{lemma}\label{lema:system}
Assume that condition ~H\ref{cond-P3}) holds. Then we have the following:
\begin{enumerate}[i)]
\item \label{cond-1-bs} all solutions $(S_n,I_n,R_n,V_n)$ of~\eqref{eq:ProblemaPrincipal}
with nonnegative initial conditions are nonnegative for all $n \in \N_0$;
\item \label{cond-2-bs} all solutions $(S_n,I_n,R_n,V_n)$ of~\eqref{eq:ProblemaPrincipal}
with positive initial conditions are positive for all $n \in \N_0$;
\item \label{cond-3-bs}  there is a constant $M > 0$ such that, if $(S_n,I_n,R_n,V_n)$ is a solution
of~\eqref{eq:ProblemaPrincipal} with nonnegative initial conditions then
\[
 \limsup_{n \to +\infty} S_n +I_n + R_n + V_n < M.
\]
\end{enumerate}
\end{lemma}

\begin{proof}
See Appendix~\ref{Appendix:A}.
\end{proof}

For each $\lambda$ and each particular solution $\xi^{*}_n=(x^{*}_n,y^{*}_n)$ of~\eqref{eq:SistemaAuxiliar} with $x^{*}_0>0$ and $y^{*}_0>0$
we define the numbers
\begin{equation}\label{eq:discrete-threshold-liminf-ell}
 \mR_D^\ell(\xi^{*},\lambda)  = \liminf_{n \to \infty} \prod_{k=n}^{n+\lambda} \frac{1+\beta_k \partial_2 \vphi(x^{*}_{k+1},0)+\sigma_k \partial_2 \psi(y^{*}_{k+1},0)}{1+\mu_k+\alpha_k+\gamma_k}
\end{equation}
and
\begin{equation}\label{eq:discrete-threshold-limsup-u}
 \mR_D^u(\xi^{*},\lambda)  = \limsup_{n \to \infty} \prod_{k=n}^{n+\lambda} \frac{1+\beta_k \partial_2 \vphi(x^{*}_{k+1},0)+\sigma_k \partial_2 \psi(y^{*}_{k+1},0)}{1+\mu_k+\alpha_k+\gamma_k},
\end{equation}
where  $\partial_i f$ denotes the partial derivative of $f$ with respect to the $i$-th variable. Contrarily to what one could expect, the next lemma shows that
the numbers above do not depend on the particular solution $\xi_n=(x_n,y_n)$ of~\eqref{eq:SistemaAuxiliar} with $x_n(0)>0$ and $y_n(0)>0$.
\begin{lemma} \label{lema:indep}
Assume that~H\ref{cond-H1}),~H\ref{cond-P1}) and~H\ref{cond-P2}) hold. If $(\xi_1^*)_n=((x_1)^*_n,(y_1)^*_n)$ and $(\xi_2^*)_n=((x_2)^*_n,(y_2)^*_n)$ are
two solutions of~\eqref{eq:SistemaAuxiliar} with $x_i^*(0)>0$ and
$y_i^*(0)>0$, $i=1,2$, then
\[
\mR_D^\ell(\xi_1^*,\lambda)=\mR_D^\ell(\xi_2^*,\lambda), \quad \text{and} \quad
\mR_D^u(\xi_1^*,\lambda)=\mR_D^u(\xi_2^*,\lambda).
\]
\end{lemma}

\begin{proof}
See Appendix~\ref{Appendix:A}.
\end{proof}

By lemma ~\ref{lema:indep} we can drop the dependence of the particular solution and simply write $\mR_D^\ell(\lambda)$ and $\mR_D^u(\lambda)$ instead of $\mR_D^\ell(\xi^*,\lambda)$ an $\mR_D^u(\xi^*,\lambda)$ respectively.

\subsection{Extinction and permanence} \label{subsection:EP}
We have the following result about the extinction of the disease:

\begin{theorem}[Extinction of the disease] \label{teo:Extinction}
Assume that conditions~H\ref{cond-H1}) to~H\ref{cond-P3}) hold. Then
\begin{enumerate}[a)]
\item \label{teo:Extinction-a} If there is a constant $\lambda>0$ such that $\mR_D^u (\lambda) <1$, then the infectives $(I_n)$ go to extinction.
\item \label{teo:Extinction-b} Any solution $(x_n^{*}, 0, 0, y_n^{*} )$, where $(x_n^*,y_n^*)$ is a particular solution of system~\eqref{eq:SistemaAuxiliar}, is globally uniformly attractive.
\end{enumerate}
\end{theorem}

\begin{proof}
See Appendix~\ref{Appendix:A}.
\end{proof}

We have the following result about the permanence of the disease:

\begin{theorem}[Permanence of the disease] \label{teo:Permanence}
Assume that conditions~H\ref{cond-H1}) to~H\ref{cond-P3}) hold. If there is a constant $\lambda>0$
such that $\mR_D^\ell(\lambda) >1$ then the infectives $(I_n)$ are permanent in system~\eqref{eq:ProblemaPrincipal}.
\end{theorem}

\begin{proof}
See Appendix~\ref{Appendix:A}.
\end{proof}

We consider now the particular periodic case: assume that all parameters of system~\eqref{eq:ProblemaPrincipal} are periodic with period $\omega \in \N$. By~\ref{cond-5-aux}) in Lemma~\ref{lema:auxSystem}, there is an $\omega$-periodic disease-free solution of~\eqref{eq:SistemaAuxiliar}, $\xi^*=(x^*_n,y^*_n)_{n \in \N}$. Thus, in the periodic setting,~\eqref{eq:discrete-threshold-liminf-ell} and~\eqref{eq:discrete-threshold-limsup-u} become both equal to
\begin{equation}\label{eq:discrete-threshold-per}
\mR_D^{per}(\xi^{*}) = \prod_{k=0}^{\omega-1} \frac{1+\beta_k \partial_2 \vphi(x^{*}_{k+1},0)+\sigma_k \partial_2 \psi(y^{*}_{k+1},0)}{1+\mu_k+\alpha_k+\gamma_k}.
\end{equation}
Therefore we obtain the corollary:
\begin{corollary}[Periodic case] \label{teo:periodic}
Assume that all coefficients are $\omega$-periodic in~\eqref{eq:ProblemaPrincipal} and that conditions~H\ref{cond-H1}) to~H\ref{cond-P3})  hold. Then
\begin{enumerate}[a)]
\item \label{teo:Extinction-per-a} If $\mR_D^{per}(\xi^{*})<1$ then the infectives $(I_n)$ go to extinction.
\item \label{teo:Extinction-per-b} The disease-free solution $(x_n^{*}, 0, 0, y_n^{*} )$, where $(x^*_n,y^*_n)_{n \in \N}$ is an disease-free $\omega$-periodic solution of~\eqref{eq:SistemaAuxiliar}, is globally attractive.
\item \label{teo:Permanence-per} If $\mR_D^{per}(\xi^{*})>1$ then the infectives $(I_n)$ are permanent.
\end{enumerate}
\end{corollary}

\begin{proof}
See Appendix~\ref{Appendix:A}.
\end{proof}

\section{Consistency} \label{section:consistency}

In this section, under the additional assumption that the parameter functions $\Lambda$, $\mu$, $\eta$ and $p$ are constant, we will get a result stating that when our integral conditions prescribe extinction (respectively persistence) for the continuous-time model, then the discrete-time conditions prescribe extinction (respectively persistence) for the corresponding discrete-time models, as long as the time step is less than some constant.
Throughout this section, we assume that the parameter functions $\Lambda$, $\mu$, $\eta$ and $p$ are constant functions and $\phi(h)$ will be the function used in the discretization of the derivative.

We consider the continuous time model~\eqref{eq:ProblemaPrincipal-cont} and, for a given time step $h$, the corresponding discrete time model, that is the discrete time model with parameters $\beta^h_k=\phi(h)\beta(kh)$, $\sigma^h_k=\phi(h)\sigma(kh)$, $\Lambda^h_k=\phi(h)\Lambda$, $\mu^h_k=\phi(h)\mu$, $p^h_k=\phi(h)p$, $\eta^h_k=\phi(h)\eta$, $\alpha^h_k=\phi(h)\alpha(kh)$ and $\gamma^h_k=\phi(h)\gamma(kh)$.

For a given time step $h>0$, the expressions $\mR_D^\ell(\lambda)$ and $\mR_D^u(\lambda)$ in~\eqref{eq:discrete-threshold-liminf-ell} and~\eqref{eq:discrete-threshold-limsup-u} become, in our context
\[
\mR_D^\ell(\lambda,h)  = \liminf_{n \to \infty} \prod_{k=n}^{n+\lambda} \frac{1+\beta^h_k \vphi(x^{*}_{k+1})+\sigma^h_k \psi(y^{*}_{k+1})}{1+\mu^h_k+\alpha^h_k+\gamma^h_k}
\]
and
\[
\mR_D^u(\lambda,h)  = \limsup_{n \to \infty} \prod_{k=n}^{n+\lambda} \frac{1+\beta^h_k \vphi(x^{*}_{k+1})+\sigma^h_k \psi(y^{*}_{k+1})}{1+\mu^h_k+\alpha^h_k+\gamma^h_k},
\]
where $(x^{*}_k,y^{*}_k)$ is the solution of the (in our context autonomous) system~\eqref{eq:SistemaAuxiliar}.

We have the following result.
\begin{theorem}\label{teo:consitency}
For system~\eqref{eq:ProblemaPrincipal-cont}, assume that $\Lambda(t)=\Lambda$, $\mu(t)=\mu$, $\eta(t)=\eta$ and $p(t)=p$ for all $t\ge 0$ and that the functions $\alpha(t)$, $\gamma(t)$, $\beta(t)$ and $\sigma(t)$ are differentiable, nonnegative, bounded and have bounded derivative. Assume also that conditions~C\ref{cond-C1}) to C\ref{cond-C3}) hold and let
$$h^u_{\text{max}}=-\dfrac{R^u_C(\lambda)}{\sup_{t \ge 0}|f'(t)|(\lambda+1)} \quad \text{and} \quad h^\ell_{\text{max}}=\dfrac{R^\ell_C(\lambda)}{\sup_{t \ge 0}|f'(t)|(\lambda+1)},$$
where
$$f(t)=\beta(t)\vphi\left(\frac{\Lambda(\mu+\eta)}{\mu(\mu+\eta+p)}\right)
+\sigma(t)\psi\left(\frac{p\Lambda}{\mu(\mu+\eta+p)}\right)-\mu-\alpha(t)-\gamma(t).$$ Then:
\begin{enumerate}[a)]
  \item \label{teo:consitency-a} If $\mR_C^u(\lambda)<0$ then $\mR_D^u(\lfloor\lambda/h\rfloor,h)<1$ for all $h \in ]0,h^u_{\text{max}}[$;
  \item \label{teo:consitency-b} If $\mR_C^\ell(\lambda)>0$ then $\mR_D^\ell(\lfloor\lambda/h\rfloor,h)>1$ for all for all $h \in ]0,h^\ell_{\text{max}}[$.
\end{enumerate}
\end{theorem}

\begin{proof}
Observe that $(x_n,y_n)=(a,b)$, $n \in \N$ and $(x(t),y(t))=(a,b)$, $t \in \R$, where
$$(a,b)=\left(\Lambda(\mu+\eta)/[\mu(\mu+\eta+p)],p\Lambda/[\mu(\mu+\eta+p)]\right),$$
are respectively solutions of system~\eqref{eq:SistemaAuxiliar} and system~\eqref{eq:SistemaAuxiliar-Pereira-Silva-Silva}. Thus
\[
\mR_C^u(\lambda) = \limsup_{t \to \infty} \int_t^{t+\lambda} \beta(s) \vphi(a)+\sigma(s) \psi(b)-[\mu(s)+\alpha(s)+\gamma(s)] \ds
\]
and
\[
\mR_D^u(\lfloor\lambda/h\rfloor,h)  = \limsup_{n \to \infty} \prod_{k=n}^{n+\lfloor\lambda/h\rfloor} \frac{1+\beta^h_k \vphi(a)+\sigma^h_k \psi(b)}{1+\mu^h_k+\alpha^h_k+\gamma^h_k}.
\]

By contradiction, assume that
\begin{equation}\label{eq:contrad-mRCu<0}
\mR_C^u(\lambda)<0,
\end{equation}
and that there is a sequence $(h_m)_{m \in \N}$ such that
$h_m \to 0$ as $m \to +\infty$ and
\begin{equation}\label{eq:contrad}
\mR_D^u(\lfloor\lambda/h_m\rfloor,h_m)  =\limsup_{n \to \infty} \prod_{k=n}^{n+\lfloor\lambda/h_m\rfloor} \frac{1+\beta^{h_m}_k \vphi(a)+\sigma^{h_m}_k \psi(b)}{1+\mu^{h_m}_k+\alpha^{h_m}_k+\gamma^{h_m}_k}\ge 1,
\end{equation}
for all $m \in \N$. By~\eqref{eq:contrad}, we conclude that, for each $m \in \N$, there are sequences $(h_m)_{m \in \N}$ and $(n_{m,r})_{r \in \N}$  such that
$h_m \to 0$ as $m \to +\infty$, $n_{m,r} \to +\infty$ as $r \to +\infty$ and
\begin{equation}\label{eq:contrad-a}
\begin{split}
& \prod_{k=n_{m,r}}^{n_{m,r}+\lfloor\lambda/h_m\rfloor}(1+\beta^{h_m}_k \vphi(a)+\sigma^{h_m}_k \psi(b))\\
& > (1-h_m) \prod_{k=n_{m,r}}^{n_{m,r}+\lfloor\lambda/h_m\rfloor}(1+\mu^{h_m}_k+\alpha^{h_m}_k+\gamma^{h_m}_k).
\end{split}
\end{equation}
By~\eqref{eq:contrad-a}, we have
\begin{equation}\label{eq:contrad-c}
\begin{split}
& \sum_{k=n_{m,r}}^{n_{m,r}+\lfloor\lambda/h_m\rfloor}(\beta^{h_m}_k \vphi(a)+\sigma^{h_m}_k \psi(b)-\mu^{h_m}_k-\alpha^{h_m}_k-\gamma^{h_m}_k) \\
& >( B_{n_{m,r},\lambda,h_m}-A_{n_{m,r},\lambda,h_m}-C_{n_{m,r},\lambda,h_m})/h_m,
\end{split}
\end{equation}
where

\begin{equation}\label{eq:An-lambda-h}
A_{n,L,h}:=-h+h\prod_{k=n}^{n+\lfloor L/h\rfloor}(1+\beta^h_k \vphi(a)+\sigma^h_k \psi(b))
-h\sum_{k=n}^{n+\lfloor L/h\rfloor}(\beta^h_k \vphi(a)+\sigma^h_k \psi(b)),
\end{equation}
\begin{equation}\label{eq:Bn-lambda-h}
B_{n,L,h}:=-h+h\prod_{k=n}^{n+\lfloor L/h\rfloor}(1+\mu^h_k+\alpha^h_k+\gamma^h_k)
-h\sum_{k=n}^{n+\lfloor L/h\rfloor}(\mu^h_k+\alpha^h_k+\gamma^h_k)
\end{equation}
and
\begin{equation}\label{eq:Cn-lambda-h}
C_{n,L,h}=
h \prod_{k=n}^{n+\lfloor L/h\rfloor}(1+\mu^{h}_k+\alpha^{h}_k+\gamma^{h}_k).
\end{equation}
and, multiplying both sides by $h_m$, we get
\Small{
\begin{equation}\label{eq:contrad-d}
\begin{split}
& \phi(h_m)\sum_{k=n_{m,r}}^{n_{m,r}+\lfloor\lambda/h_m\rfloor}h_m\left[\beta(kh_m) \vphi(a)+\sigma(kh_m) \psi(b)-\mu(kh_m)-\alpha(kh_m)-\gamma(kh_m)\right]\\
& > B_{n_{m,r},\lambda,h_m}-A_{n_{m,r},\lambda,h_m}-C_{n_{m,r},\lambda,h_m}.
\end{split}
\end{equation}
}

We also have
\begin{equation}\label{eq:maj-A-nlh-prev}
\begin{split}
|A_{n_{m,r},\lambda,h_m}|
& \le h_m \sum_{k=2}^{\lfloor\lambda/h_m\rfloor}  \binom{\lfloor\lambda/h_m\rfloor}{k}[(\beta^u \vphi(a)+\sigma^u \psi(b))]^k[\phi(h_m)]^k\\
& \le h_m \sum_{k=0}^{\lfloor\lambda/h_m\rfloor}  \binom{\lfloor\lambda/h_m\rfloor}{k}[(\beta^u \vphi(a)+\sigma^u \psi(b))]^k[\phi(h_m)]^k\\
& = h_m \left[1+(\beta^u \vphi(a)+\sigma^u \psi(b))\phi(h_m)\right]^{\lfloor\lambda/h_m\rfloor}.
\end{split}
\end{equation}
Noting that, by~\eqref{eq:derivative-denominator}, we have
\[
\begin{split}
& \lim_{m \to +\infty} \left[1+(\beta^u \vphi(a)+\sigma^u \psi(b))\phi(h_m)\right]^{\lfloor\lambda/h_m\rfloor}\\
& = \lim_{m \to +\infty} \left[\left(1+\frac{\beta^u \vphi(a)+\sigma^u \psi(b)}{1/\phi(h_m)}\right)^{1/\phi(h_m)}\right]^{\phi(h_m)\lfloor\lambda/h_m\rfloor}\\
& = \e^{(\beta^u \vphi(a)+\sigma^u \psi(b))\lambda}\\
\end{split}
\]
and that a convergent sequence is bounded, by~\eqref{eq:maj-A-nlh-prev} there is $C_1>0$ such that
\begin{equation}\label{eq:maj-A-nlh}
|A_{n_{m,r},\lambda,h_m}|
\le C_1 h_m.
\end{equation}
Similarly, we have
\begin{equation}\label{eq:maj-B-nlh-prev}
\begin{split}
|B_{n_{m,r},\lambda,h_m}|
& \le \sum_{k=2}^{\lfloor\lambda/h_m\rfloor}  \binom{\lfloor\lambda/h_m\rfloor}{k}[(\mu^u+\alpha^u+\gamma^u)]^k[\phi(h_m)]^k\\
& \le \sum_{k=0}^{\lfloor\lambda/h_m\rfloor}  \binom{\lfloor\lambda/h_m\rfloor}{k}[(\mu^u+\alpha^u+\gamma^u)]^k[\phi(h_m)]^k\\
& = h_m \left[1+(\mu^u+\alpha^u+\gamma^u)\phi(h_m)\right]^{\lfloor\lambda/h_m\rfloor}.
\end{split}
\end{equation}
Using~\eqref{eq:derivative-denominator} again, we get
\[
\lim_{m \to +\infty} \left[1+(\mu^u+\alpha^u+\gamma^u)\phi(h_m)\right]^{\lfloor\lambda/h_m\rfloor}
\le \e^{(\mu^u+\alpha^u+\gamma^u)\lambda},
\]
there is $C_2>0$ such that
\begin{equation}\label{eq:maj-B-nlh}
|B_{n_{m,r},\lambda,h_m}| \le C_2 h_m.
\end{equation}
Finally, we have
\[
\begin{split}
|C_{n_{m,r},\lambda,h_m}|
& = h_m \prod_{k=n_{m,r}}^{n_{m,r}+\lfloor \lambda/h_m\rfloor}(1+\mu^{h_m}_k+\alpha^{h_m}_k+\gamma^{h_m}_k)\\
& = h_m (1+3\phi(h_m)\max\{\mu^u,\alpha^u,\gamma^u\})^{\lfloor \lambda/h_m\rfloor+1}.
\end{split}
\]
According~\eqref{eq:derivative-denominator}, we obtain
\[
\begin{split}
& \lim_{m \to +\infty} (1+3\phi(h_m)\max\{\mu^u,\alpha^u,\gamma^u\})^{\lfloor \lambda/h_m\rfloor+1}\\
& =\lim_{m \to +\infty} \left[\left(1+\frac{3\max\{\mu^u,\alpha^u,\gamma^u\}}{1/\phi(h_m)}\right)^{1/\phi(h_m)}\right]^{\phi(h_m)(\lfloor \lambda/h_m\rfloor+1)}\\
& = \e^{3\max\{\mu^u,\alpha^u,\gamma^u\}\lambda},
\end{split}
\]
there is $C_3>0$ such that
\begin{equation}\label{eq:maj-C-nlh}
|C_{n_{m,r},\lambda,h_m}| \le C_3 h_m.
\end{equation}

Thus
\begin{equation}\label{eq:to-zero-uniformly}
\begin{split}
& B_{n_{m,r},\lambda,h_m}-A_{n_{m,r},\lambda,h_m}-C_{n_{m,r},\lambda,h_m}\\
& \le |A_{n_{m,r},\lambda,h_m}|+|B_{n_{m,r},\lambda,h_m}|+|C_{n_{m,r},\lambda,h_m}|\\
& \le (C_1+C_2+C_3)h_m,
\end{split}
\end{equation}
for all $m \ge M$. Since the right hand side of~\eqref{eq:to-zero-uniformly} is independent of $n_{m,r}$, we conclude that
\begin{equation}\label{eq:(An-Bn)/phi-to-0}
B_{n_{m,r},\lambda,h_m}-A_{n_{m,r},\lambda,h_m}-C_{n_{m,r},\lambda,h_m} \to 0,
\end{equation}
as $m \to +\infty$, uniformly in $r$.

On the other hand we note that the $C^1$ function $f:\R^+_0 \to \R$ given by
\begin{equation}\label{eq:funct-integ}
f(t)=\beta(t) \vphi(a)+\sigma(t) \psi(b)-\mu(t)-\alpha(t)-\gamma(t)
\end{equation}
is Riemann-integrable on any bounded interval $I \subset \R_0^+$.

We have that
\[
\begin{split}
\sum_{k=n_{m,r}}^{n_{m,r}+\lfloor\lambda/h_m\rfloor} h_mf(kh_m)
+(\lambda-\lfloor\lambda/h_m\rfloor h_m)f(n_{m,r} h_m+\lfloor\lambda/h_m\rfloor h_m),
\end{split}
\]
is a Riemann sum of
$$\int_{n_{m,r} h_m}^{n_{m,r} h_m+\lambda} \beta(s) \vphi(a)+\sigma(s) \psi(b)-[\mu(s)+\alpha(s)+\gamma(s)] \ds$$
with respect to the partition
$$\{n_{m,r}h_m,n_{m,r}h_m+h_m,\ldots,n_{m,r}h_m+\lfloor\lambda/h_m\rfloor h_m,n_{m,r}h_m+\lambda\}$$
of size $h_m$ of the interval $[n_{m,r},n_{m,r}+\lambda]$. Note that
\[
s_{m,r}:=(\lambda-\lfloor\lambda/h_m\rfloor h_m)f(n_{m,r}h_m+\lfloor\lambda/h_m\rfloor h_m) \le h_mf^u :=s_m
\]
and $s_m \to 0$ as $m \to +\infty$, uniformly in $r$.

Since $f$ is $C^1$ with bounded derivative, for any $h>0$  we have
    $$|f(x)-f(x+h)|\le Ch,$$
where $\displaystyle C=\sup_{t\ge 0} |f'(t)|$. We conclude that
    \begin{equation}\label{eq:sum-int-to-0}
    \begin{split}
    &\left|\sum_{k=n_{m,r}}^{n_{m,r}+\lfloor\lambda/h_m\rfloor} h_mf(kh_m)+s_{m,r} - \int_{n_{m,r} h_m}^{n_{m,r} h_m+\lambda} f(s)\ds\right|\\
    & < Ch_m^2\lfloor\lambda/h_m \rfloor+Ch_m^2\\
    & < C(\lambda+1)h_m,
    \end{split}
    \end{equation}
    thus
    \[
    \sum_{k=n_{m,r}}^{n_{m,r}+\lfloor\lambda/h_m\rfloor} h_mf(kh_m)<\int_{n_{m,r} h_m}^{n_{m,r} h_m+\lambda} f(s)\ds-s_{m,r}+C(\lambda+1)h_m,
    \]
    and therefore
    \Small{
    \begin{equation}\label{eq:aprox-int-sum}
    \sum_{k=n_{m,r}}^{n_{m,r}+\lfloor\lambda/h_m\rfloor} \phi(h_m)h_mf(kh_m)<\phi(h_m)\left[\int_{n_{m,r} h_m}^{n_{m,r} h_m+\lambda} f(s)\ds+C(\lambda+1)h_m\right].
    \end{equation}
    }
By~\eqref{eq:aprox-int-sum} we conclude that, given $\delta >0$, there is $r_m \in \N$ such that, for all $r\ge r_m$,
\begin{equation}\label{eq:contradiction}
\phi(h_m) \sum_{k=n_{m,r_m}}^{n_{m,r_m}+\lfloor\lambda/h_m\rfloor} h_mf(kh_m)<\phi(h_m)\left[\mR_C^u(\lambda)+\delta+C(\lambda+1)h_m\right].
\end{equation}

Finally, recalling that $\mR_C^u(\lambda)<0$, by assumption, by the arbitrariness of $\delta>0$ and the fact that $h_m \to 0$ as $m \to +\infty$, we obtain for sufficiently large $m \in \N$,
$$0 \le \phi(h_m) \sum_{k=n_{m,r_m}}^{n_{m,r_m}+\lfloor\lambda/h_m\rfloor} h_mf(kh_m)<0,$$
which is a contradiction. We obtain~\ref{teo:consitency-a}).

A similar argument allow us to prove~\ref{teo:consitency-b}). In fact, assuming by contradiction that
\begin{equation}\label{eq:contrad-mRCu>0}
\mR_C^\ell(\lambda)>0,
\end{equation}
and that there is a sequence $(h_m)_{m \in \N}$ such that
$h_m \to 0$ as $m \to +\infty$ and
\[
\mR_D^\ell(\lfloor\lambda/h_m\rfloor,h_m)  =\liminf_{n \to \infty} \prod_{k=n}^{n+\lfloor\lambda/h_m\rfloor} \frac{1+\beta^{h_m}_k \vphi(a)+\sigma^{h_m}_k \psi(b)}{1+\mu^{h_m}_k+\alpha^{h_m}_k+\gamma^{h_m}_k} \le 1,
\]
it is possible to conclude that
\begin{equation}\label{eq:contrad-d--per}
\begin{split}
& \phi(h_m)\sum_{k=n_{m,r}}^{n_{m,r}+\lfloor\lambda/h_m\rfloor}h_m(\beta(kh_m) \vphi(a)+\sigma(kh_m) \psi(b)-\mu(kh_m)-\alpha(kh_m)-\gamma(kh_m))\\
& < B_{n_{m,r},\lambda,h_m}-A_{n_{m,r},\lambda,h_m}+C_{n_{m,r},\lambda,h_m},
\end{split}
\end{equation}
where $A_{n_{m,r},\lambda,h_m}$, $B_{n_{m,r},\lambda,h_m}$ and $C_{n_{m,r},\lambda,h_m}$ are given respectively by~\eqref{eq:An-lambda-h},~\eqref{eq:Bn-lambda-h} and~\eqref{eq:maj-C-nlh} and still satisfy~\eqref{eq:maj-A-nlh},~\eqref{eq:maj-B-nlh} and~\eqref{eq:maj-C-nlh}. Consequently, given $\delta >0$, there is $r_m \in \N$ such that, for all $r\ge r_m$,
    \begin{equation}\label{eq:aprox-int-sum-per}
    \phi(h_m)\sum_{k=n_{m,r}}^{n_{m,r}+\lfloor\lambda/h_m\rfloor} h_mf(kh_m)>\phi(h_m)\left[\int_{n_{m,r} h_m}^{n_{m,r}h_m+\lambda} f(s)\ds-\delta+C(\lambda+1)h_m\right].
    \end{equation}
Recalling that $\mR_C^\ell(\lambda)>0$, by assumption and since $\delta>0$ is arbitrary, we obtain for sufficiently large $m \in \N$,
$$0 \ge \phi(h_m)\sum_{k=n_{m,r_m}}^{n_{m,r_m}+\lfloor\lambda/h_m\rfloor} f(kh_m)>0,$$
which is a contradiction. We obtain~\ref{teo:consitency-b}) and the theorem follows.
\end{proof}

Next, for each $L \in \N$, we give an example of a periodic system of period $1$ such that the continuous and the discrete time system with time step $h=1/L$ are not consistent, namely we will have persistence for the continuous time model and extinction for the discrete time model with time step $h=1/L$.

\begin{example}
Let $L \in \N$. Consider in system~\eqref{eq:ProblemaPrincipal-cont} that $\phi(x)=\psi(x)=x$, that, with the exception of $\sigma$ and $\beta$, all parameters are constant, that $\Lambda=\mu$ and that
    $$\sigma(t)=\beta(t)=d[1+c\sin^2(2\pi L t)(1+\cos(2\pi t))].$$
We obtain a periodic system of period $1$.

In this context, $(x_n,y_n)=(a,b)$, $n \in \N$, and $(x(t),y(t))=(a,b)$, $t \in \R$, where
$$(a,b)=\left((\mu+\eta)/(\mu+\eta+p),p/\mu+\eta+p)\right),$$
are respectively solutions of system~\eqref{eq:SistemaAuxiliar} and system~\eqref{eq:SistemaAuxiliar-Pereira-Silva-Silva}.
It is now possible to compute the number $\mathcal R_C^\ell(1)$. In fact, noting that $x^*(t)+y^*(t)=1$, we get
\[
\begin{split}
\mathcal R_C^\ell(1)
& =\int_0^1 \beta(s)x^*(s)+\sigma(s)y^*(s)-\mu-\alpha-\gamma \ds \\
& =\int_0^1 d[1+c\sin^2(2\pi L t)(1+\cos(2\pi t))] \ds -\mu-\alpha-\gamma\\
& =d(1+c/2)-\mu-\alpha-\gamma.
\end{split}
\]
We can also compute $\mathcal R_D^\ell(1,1/L)$. Namely we have
$$\mathcal R_D^\ell(1,1/L)=\frac{1+d/L}{1+\mu+\alpha+\gamma}.$$

If we let $d$ be sufficiently small so that $d<(\mu+\alpha+\gamma)L$, or in other words, $d<(\mu+\alpha+\gamma)$ and $c$ be sufficiently large so that $c>\frac{2}{d}(\mu+\gamma+\alpha-d)$, we obtain

$$\mathcal R_C^\ell(1)>1 \quad \Leftrightarrow \quad \frac{1+d(1+c/2)}{1+\mu+\alpha+\gamma}>1$$
and
$$\mathcal R_D^\ell(1,1/L)<1 \quad \Leftrightarrow \quad \frac{1+d/L}{1+\mu+\alpha+\gamma}<1.$$
so we conclude that we don't have consistency for  time step $1/L$.

Let $L=6$ and consider the continuous model with the following parameters $\mu=\Lambda=0.25$, $\gamma=0.3$, $\alpha=0.05$, $\eta=0.05$, $p=2/3$, $d=0.6$ and $c=1.5$. In figure~\ref{fig-exemplo}, we plot function $\beta$ (or similarly $\sigma$) and the component $I(t)$ of the solution of system~\eqref{eq:ProblemaPrincipal-cont} given by the solver of Mathematica$^\circledR$ (that we take to represent the solution of the continuous-time model) and the solution of the discrete-time model~\eqref{eq:ProblemaPrincipal-disc} with time step $1/6$. As can be seen, the infectives are persistent in the continuous-time model but go to extinction in the discrete-time model. We have inconsistency in this case.
\begin{figure}[h]
  \begin{minipage}[b][2.7cm]{.4\linewidth}
    \includegraphics[scale=0.4]{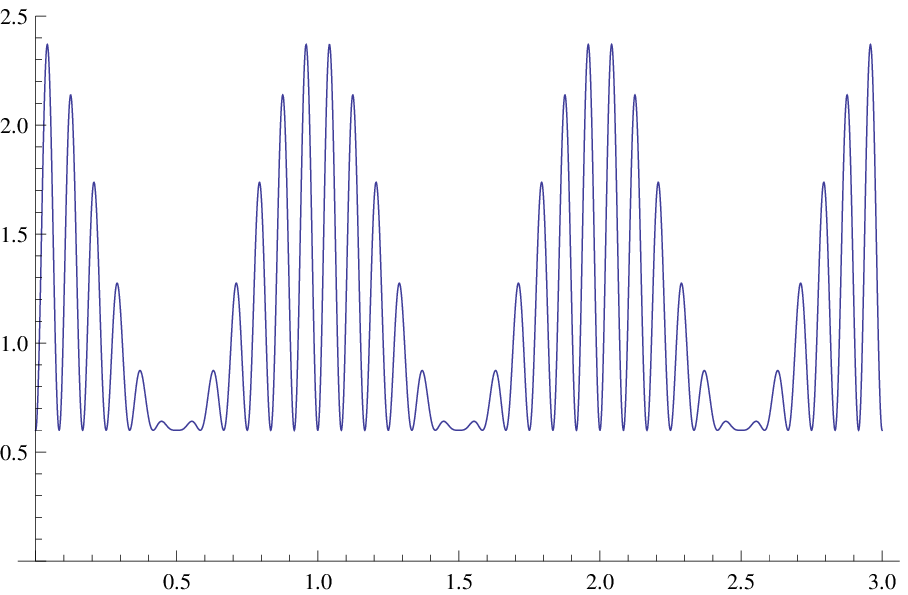}
  \end{minipage}
  \begin{minipage}[b][2.7cm]{.2\linewidth}
    \includegraphics[scale=0.4]{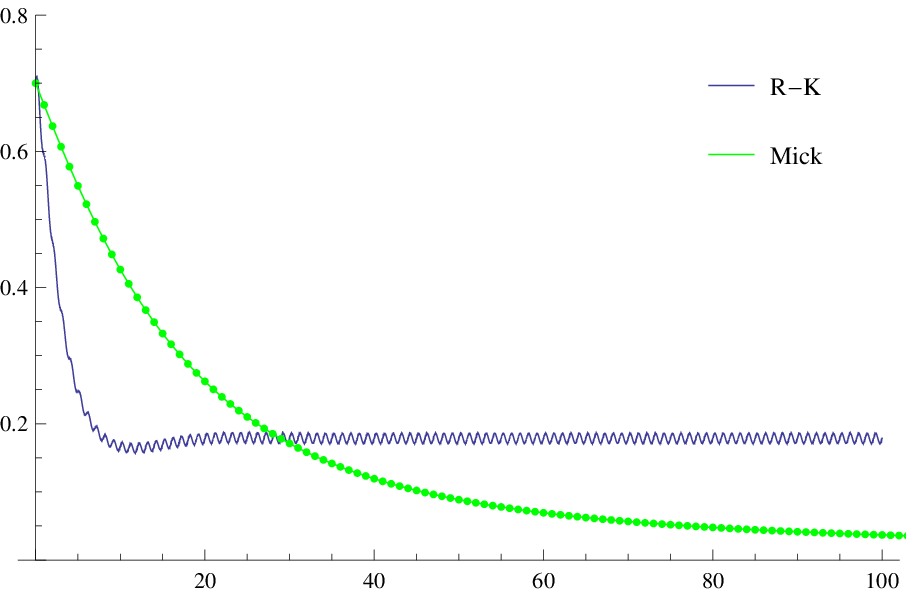}
  \end{minipage}
\caption{Left: function $\beta$; right: inconsistency (time-step=1/6).}
\label{fig-exemplo}
\end{figure}
\end{example}

Note that, changing $\beta(t)$ and $\sigma(t)$ slightly, we can construct an example of a periodic system with period $1$  where the infectives in the continuous time model goes to extinction but, in the discrete time model with time step $h=1/L$, the infectives are persistent.

Furthermore, we emphasize that this lack of consistency is not a result of the discretization method used but simply a result of the fact that the time steps lead to a situation where the points $n/L$ where the functions $\beta$ and $\sigma$ are evaluated (in order to obtain the discrete time parameters) correspond to minimums of $\beta$ and $\sigma$.
\section{Simulation} \label{section:simulation}
Our objective in this section is twofold. On the one hand, we want to consider different incidence functions $\varphi$, corresponding to different discretizations of our continuous model, and compare the several discrete models obtained. We do this in the first subsection. On the other hand, we want to use our model to describe a real situation. We do this in the second subsection where we consider data from the incidence of measles in France in the period 2012-2016.

\subsection{Simulation with several NSFD schemes} \label{subsection:simulation1}
In this subsection we do some simulation to illustrate our results. To begin, we compare our model~\eqref{eq:ProblemaPrincipal} with mass action incidence ($\vphi(S,I)=SI$ and $\psi(V,I)=VI$) with Zhang's model~\cite{Zhang-AMC-2015}. We use the following set of parameters: $\phi(h)=h+0.2 h^2$,
$\Lambda=0.5$, $\mu(t)=\gamma(t)=\delta(t)=0.3$, $\alpha(t)=0.05$, $\eta=0.05$, $p=2/3$ and
    $$\beta(t)=\sigma(t)=b(1+0.3\cos(t\pi/2)).$$
Setting $b=0.3$ we obtain $\mathcal R^u_C(4)=-0.6<0$ and thus we conclude that we have extinction for the continuous model. Taking time-steps equal to $4$, $1$ and $0.5$, we get $\mathcal R_D^u(0,4)=\mathcal R_D^\ell(0,4)=1$, $\mathcal R_D^u(3,1)=0.644<1$ and $\mathcal R_D(7,0.5)=0.601<1$ and we conclude that we have extinction for time steps $1$ and $0.5$. For these parameters, we have consistency in the sense of Theorem~\ref{teo:consitency} as long as the time step is less than $0.05$. Clearly, there is numerical evidence that there is consistency even for higher time steps.
Figure~\ref{fig-compare-Mick-MickZhang-ext} illustrates this situation.
\begin{figure}[h]
  \begin{minipage}[b][2.7cm]{.3\linewidth}
    \includegraphics[scale=0.4]{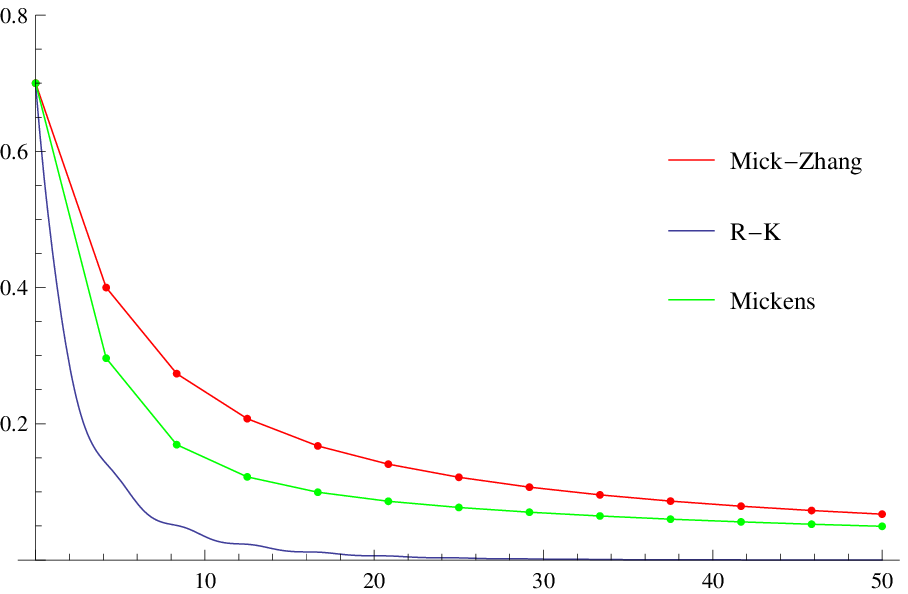}
  \end{minipage}
  \begin{minipage}[b][2.7cm]{.3\linewidth}
        \includegraphics[scale=0.4]{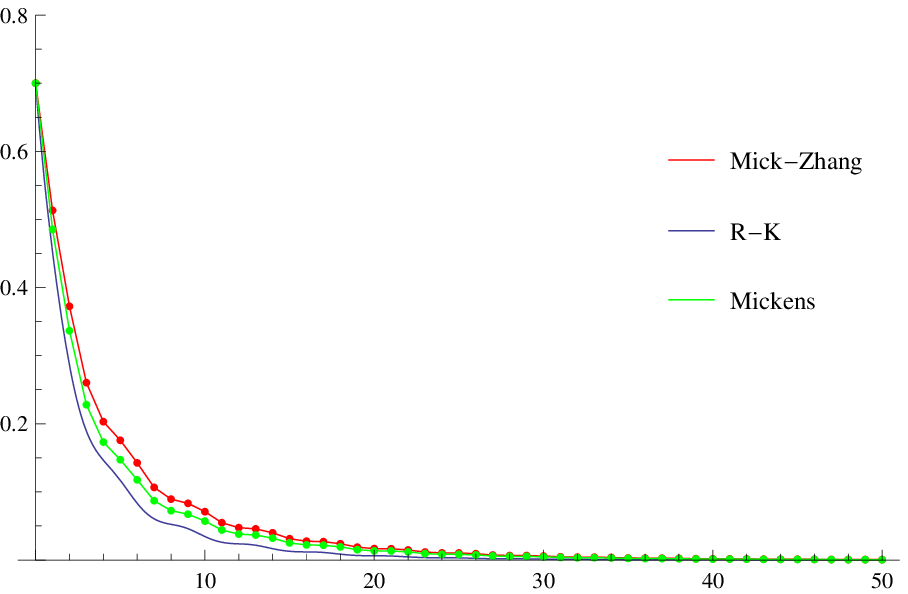}
  \end{minipage}
  \begin{minipage}[b][2.7cm]{.3\linewidth}
        \includegraphics[scale=0.4]{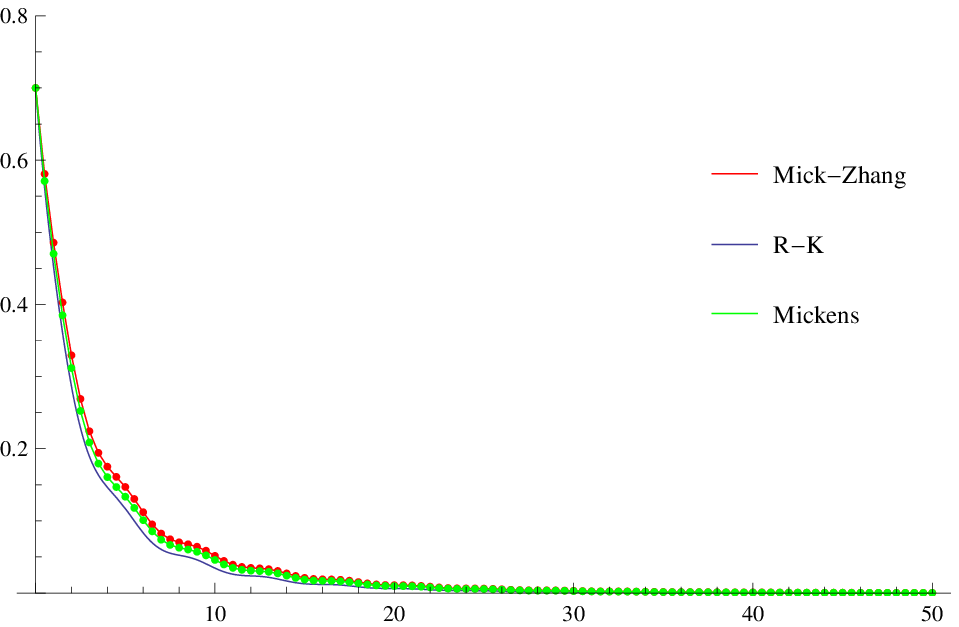}
  \end{minipage}
\caption{$SI$; $\phi(h)=h+0.2 h^2$; step-size: $4$, $1$, $0.5$.}
\label{fig-compare-Mick-MickZhang-ext}
\end{figure}

Changing $b$ to $0.9$ we obtain $\mathcal R^\ell_C(4)=3.4>0$ and thus we conclude that we have persistence.
Taking time-steps equal to $2$, $1$ and $0.5$, we get $\mathcal R_D^\ell(1,2)=3.201>1$, $\mathcal R_D^\ell(3,1)=5.9>1$ and $\mathcal R_D^\ell(7,0.5)=10.2>1$ and we conclude that we have persistence for all these time steps. Figure~\ref{fig-compare-Mick-MickZhang-per} illustrates this situation. Figure~\ref{fig-compare-Mick-MickZhang-ext} and figure~\ref{fig-compare-Mick-MickZhang-per} suggest that numerically our model is slightly better that Zhang's model, at least for large time steps.

\begin{figure}[h]
  \begin{minipage}[b][2.7cm]{.3\linewidth}
    \includegraphics[scale=0.4]{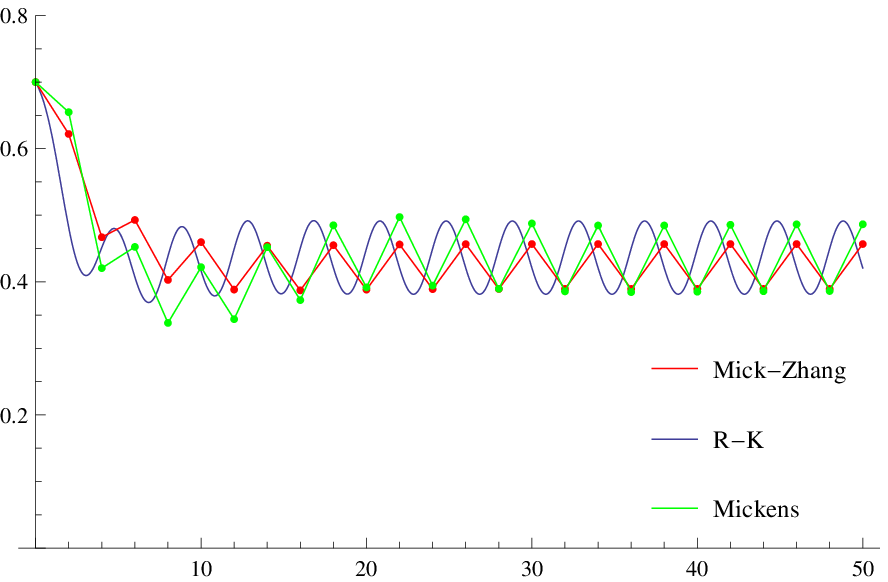}
  \end{minipage}
  \begin{minipage}[b][2.7cm]{.3\linewidth}
        \includegraphics[scale=0.4]{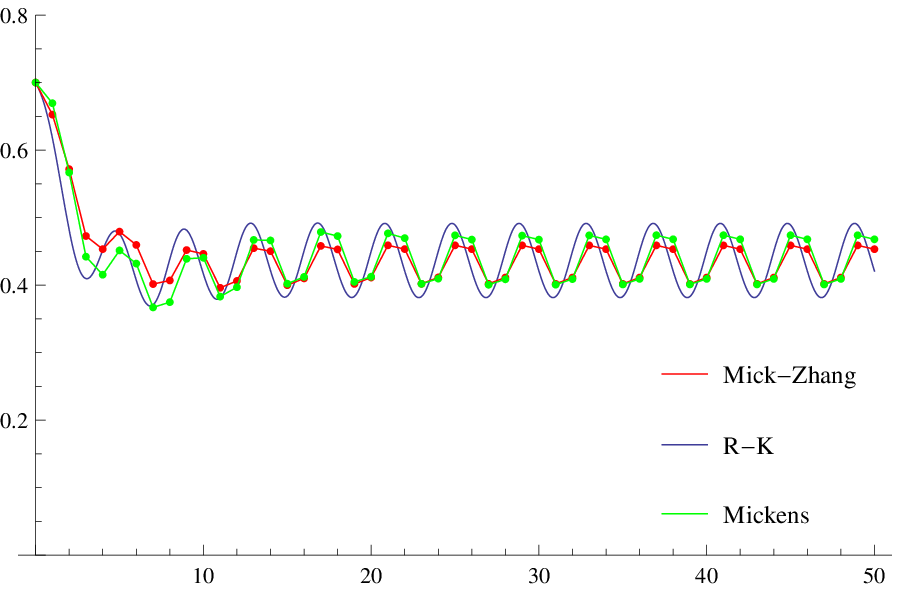}
  \end{minipage}
  \begin{minipage}[b][2.7cm]{.3\linewidth}
        \includegraphics[scale=0.4]{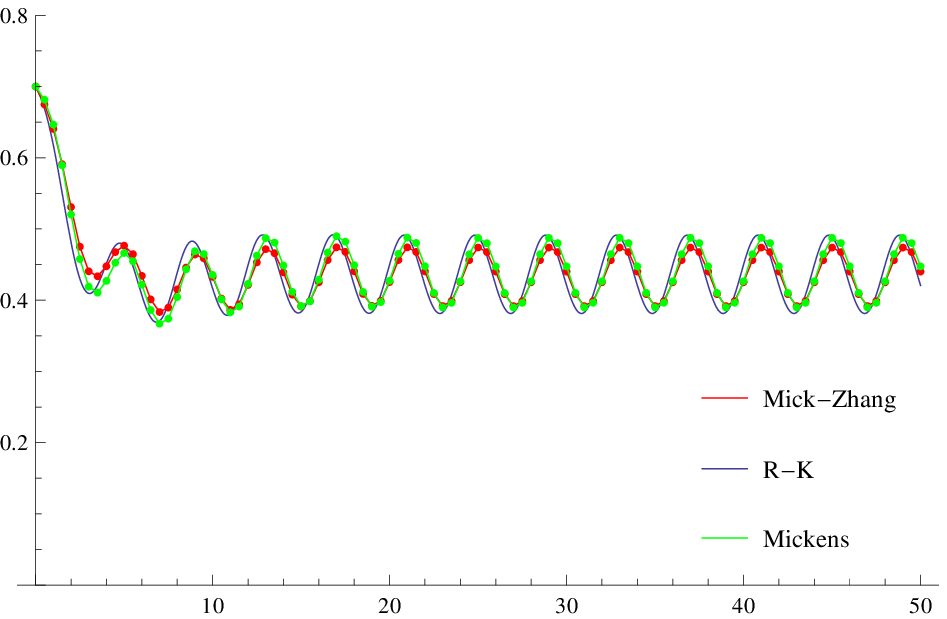}
  \end{minipage}
\caption{$SI$; $\phi(h)=h+0.2 h^2$; step-size: $2$, $1$, $0.5$.}
\label{fig-compare-Mick-MickZhang-per}
\end{figure}

Next, we compare our model with the discretized model obtained by Euler method and the output of the Mathematica$^\circledR$ solver ODE (that uses a Runge-Kutta method). Considering $b=0.3$, we get extinction for the continuous time model, as we already saw. Taking time steps equal to $2$, $1$ and $0.5$, we can see in figure~\ref{fig-compare-Mick-Euler-RK-ext} that for all methods considered and all time steps we have extinction, although the behaviour of our model shadows better the behaviour given by Mathematica's solver, at least for these time steps.
\begin{figure}[h]
  \begin{minipage}[b][2.7cm]{.3\linewidth}
    \includegraphics[scale=0.4]{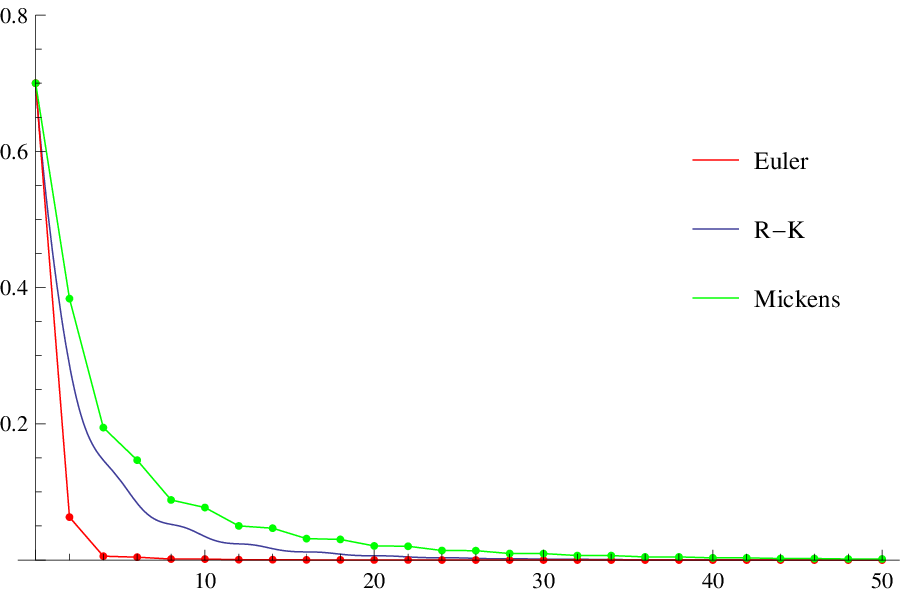}
  \end{minipage}
  \begin{minipage}[b][2.7cm]{.3\linewidth}
        \includegraphics[scale=0.4]{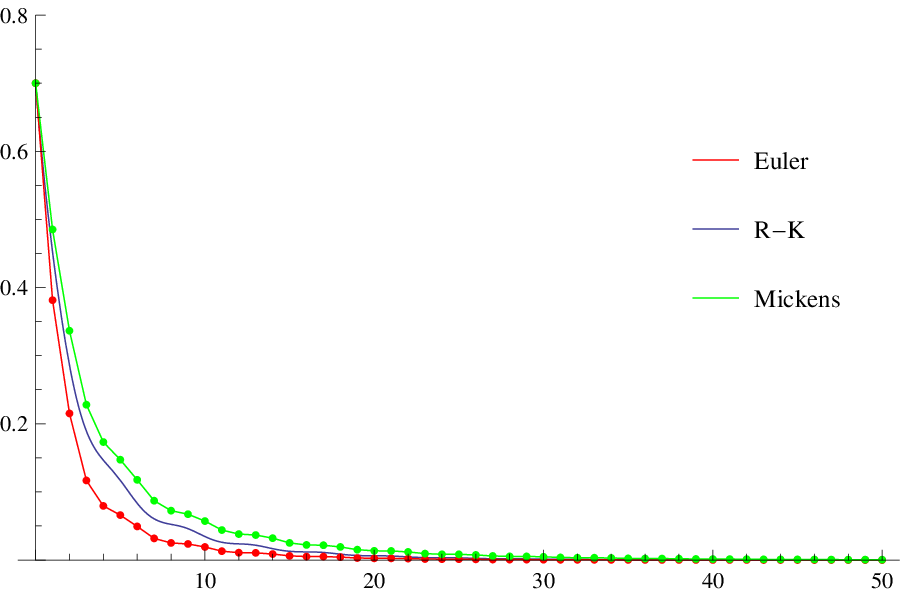}
  \end{minipage}
  \begin{minipage}[b][2.7cm]{.3\linewidth}
        \includegraphics[scale=0.4]{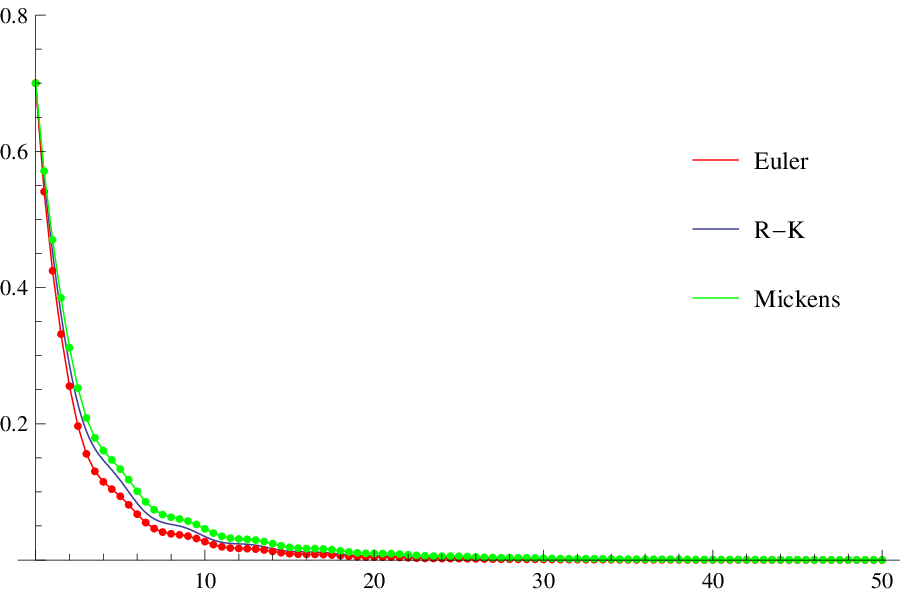}
  \end{minipage}
\caption{$SI$; $\phi(h)=h+0.2 h^2$; step-size: $2$, $1$, $0.5$.}
\label{fig-compare-Mick-Euler-RK-ext}
\end{figure}

Changing $b$ to $0.9$ we already saw that we get persistence for the continuous model.
Figure~\ref{fig-compare-Mick-Euler-RK-per} illustrates this situation.
\begin{figure}[h]
  \begin{minipage}[b][2.7cm]{.3\linewidth}
    \includegraphics[scale=0.4]{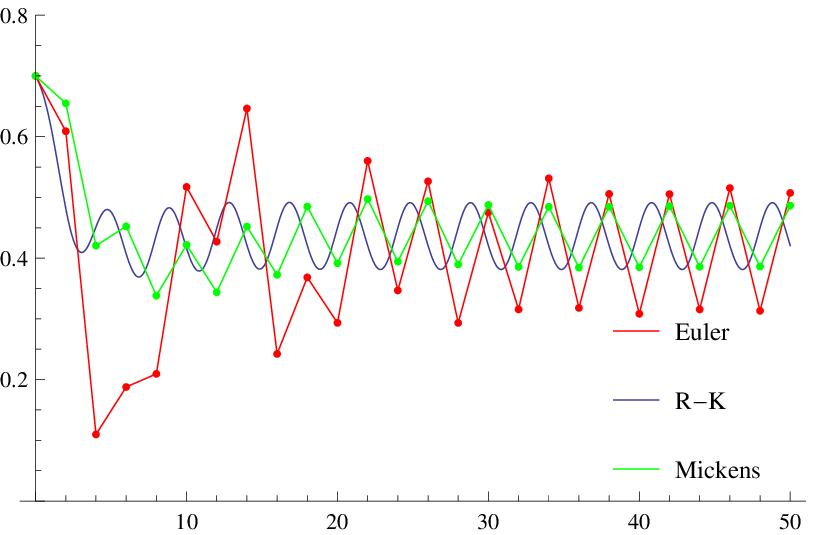}
  \end{minipage}
  \begin{minipage}[b][2.7cm]{.3\linewidth}
        \includegraphics[scale=0.4]{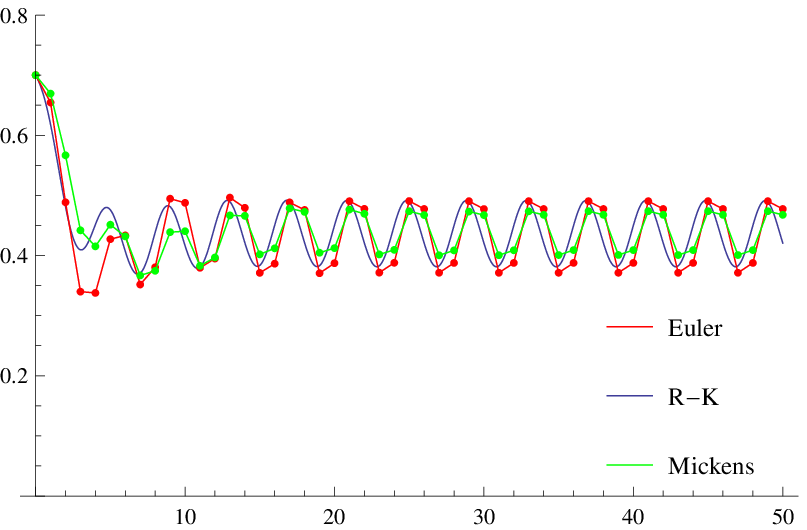}
  \end{minipage}
  \begin{minipage}[b][2.7cm]{.3\linewidth}
        \includegraphics[scale=0.4]{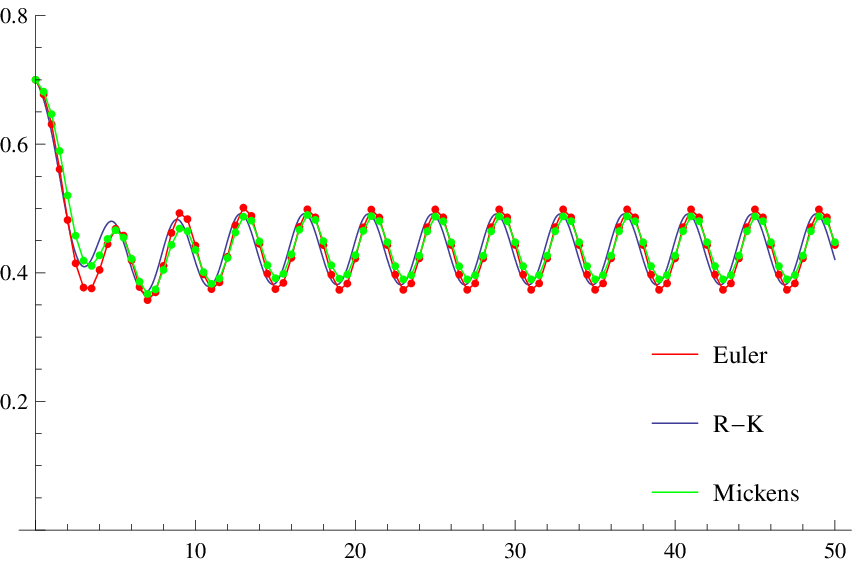}
  \end{minipage}
\caption{$SI$; $\phi(h)=h+0.2 h^2$; step-size: $2$, $1$, $0.5$.}
\label{fig-compare-Mick-Euler-RK-per}
\end{figure}

Next, we change our incidence function and consider $\phi(S,I)=SI/(1+0.7I)$, maintaining the set of parameters. Letting $b=0.3$ we have extinction for the continuous model and letting $b=0.9$ we have persistence for the continuous time model. Note that the thresholds $\mathcal R_C^u$, $\mathcal R_C^\ell$, $\mathcal R_D^u$ and $\mathcal R_D^\ell$ are similar to the mass action case. Figures~\ref{fig-compare-Mick-Euler-RK-ext-Hol} and~\ref{fig-compare-Mick-Euler-RK-per-Hol} illustrate this situation.
\begin{figure}[h]
  \begin{minipage}[b][2.7cm]{.3\linewidth}
    \includegraphics[scale=0.4]{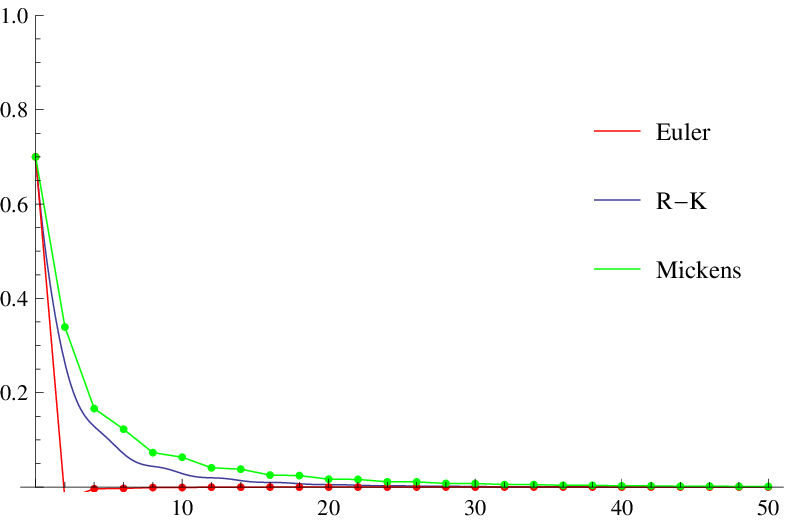}
  \end{minipage}
  \begin{minipage}[b][2.7cm]{.3\linewidth}
        \includegraphics[scale=0.4]{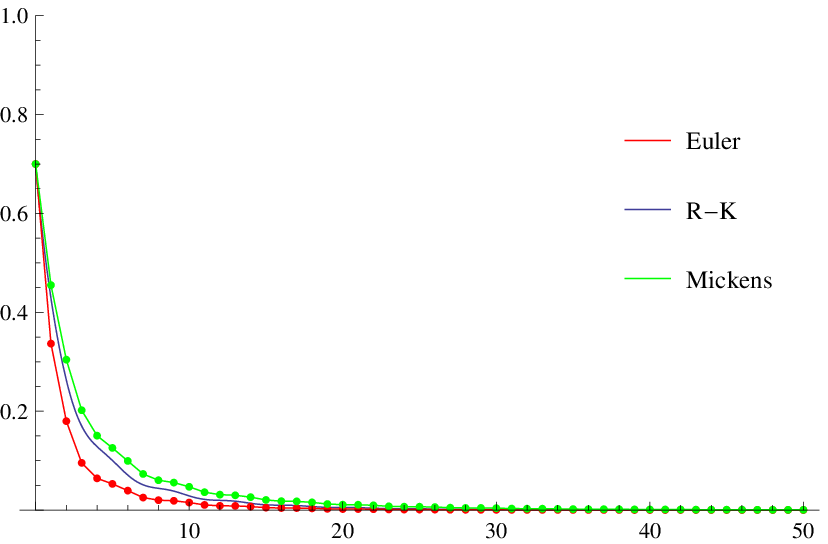}
  \end{minipage}
  \begin{minipage}[b][2.7cm]{.3\linewidth}
        \includegraphics[scale=0.4]{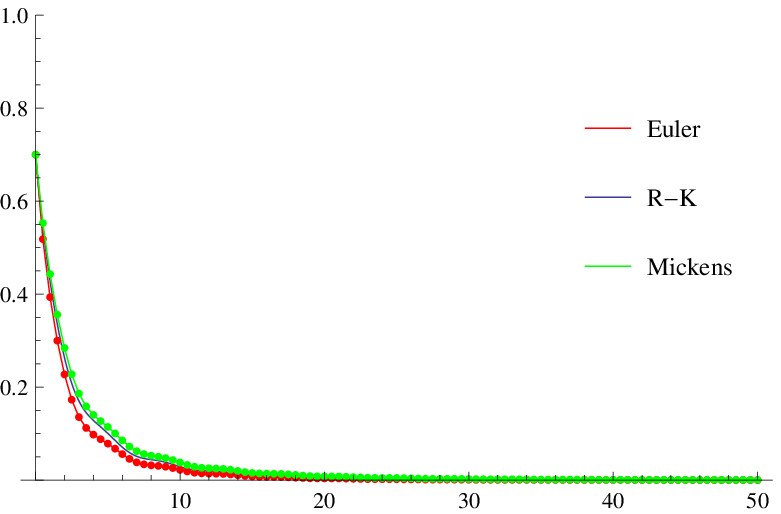}
  \end{minipage}
\caption{$\frac{SI}{1+0.7I}$; $\phi(h)=h+0.2 h^2$; step-size: $2$, $1$, $0.5$.}
\label{fig-compare-Mick-Euler-RK-ext-Hol}
\end{figure}
\begin{figure}[h]
  \begin{minipage}[b][2.7cm]{.3\linewidth}
    \includegraphics[scale=0.4]{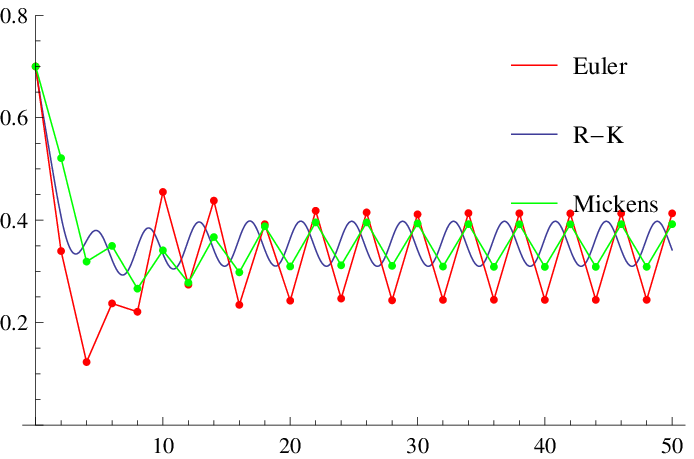}
  \end{minipage}
  \begin{minipage}[b][2.7cm]{.3\linewidth}
        \includegraphics[scale=0.4]{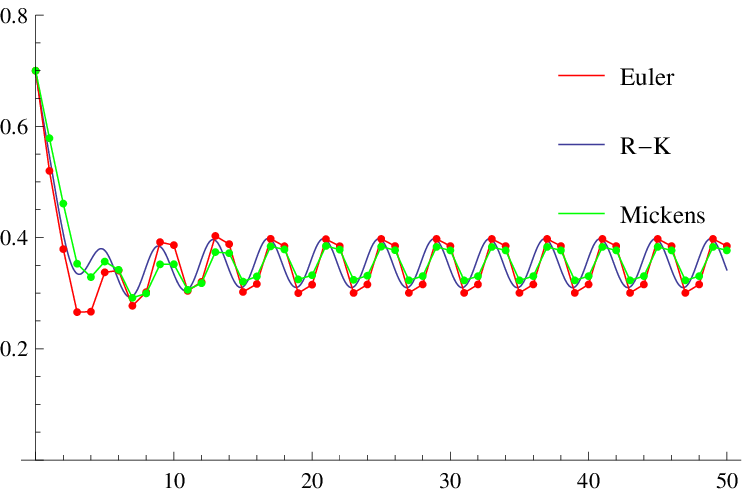}
  \end{minipage}
  \begin{minipage}[b][2.7cm]{.3\linewidth}
        \includegraphics[scale=0.4]{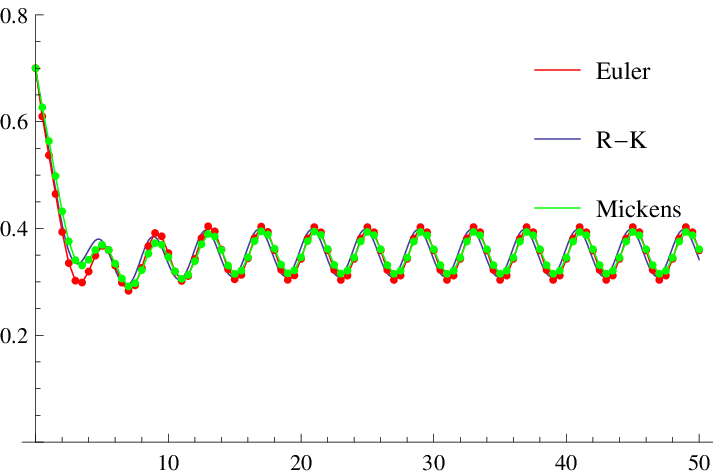}
  \end{minipage}
\caption{$\frac{SI}{1+0.7I}$; $\phi(h)=h+0.2h^2$; step-size: $2$, $1$, $0.5$.}
\label{fig-compare-Mick-Euler-RK-per-Hol}
\end{figure}

Doing corresponding simulations and comparisons for our model with $\phi(h)=(1-\e^{-0.002h})/(0.002)$ instead of $\phi(h)=h+0.2h^2$ we can draw the same conclusions regarding extinction/persistence, relation to Zhang's model and the model obtained by Euler method.

\subsection{Simulation with real data} \label{subsection:simulation2}
In this subsection, we present some simulation regarding measles. This disease is endemic in some countries such as France. In that country, with the measles outbreak in 2011, it was introduced a vaccination policy that lowered the number of reported cases. We will focus on measles in France, between 2012-2016. For a study concerning the period before 2012 see \cite{Bacaer-2014}. For our parameters estimation, we gathered information from several websites. We considered standard incidence functions $\psi(V_{n+1}, I_n) =  {V_{n+1}I_{n}}/{P_n}$ and $\varphi(S_{n+1}, I_n)={S_{n+1}I_n}/{P_n}$, where $P_n$ is the total population. Inspired in the time series for the infectives (https://ecdc.europa.eu), we considered $\sigma_n=0.03$ and $\beta_n$ given by

\[
 \beta_n=\left\{
            \begin{array}{ll}
              3.8+10 \sin\left(\frac{(n+1)\pi}{6}\right), & \hbox{if $\Bigl \lfloor\frac{n}{12} \Bigr \rfloor \leq 5$} \\
              2.7, & \hbox{otherwise}
            \end{array}.
          \right.
\]
The remaining parameters were considered time independent and were inspired in data contained in the websites www.worldbank.org, https://data.oecd.org and   www.geoba.se. Namely, we took the mortality rate $\mu_n=0.0007$, the newborns $\Lambda_n=50000$, the disease induced mortality $\alpha_n=0.000375$,  the immunity loss $\eta_n=0.001$,  the vaccination rate $p_n=0.001$  and the recovery rate $\gamma_n=0.957$. We used the initial conditions $S_0=7.20428\times10^6$, $I_0=106$, $V_0=5.84372\times10^7$ and $R_0=1.81918\times10^4$. In Figure~\ref{fig-Measles-Simul} we plot the real data for the infectives and the output given by our model.

\begin{figure}[h]
\centering
    \includegraphics[scale=0.7]{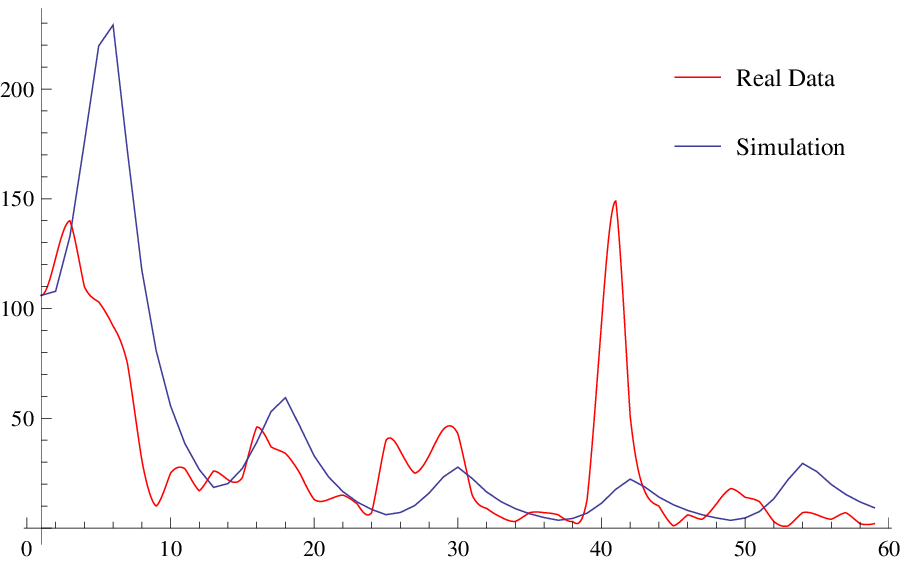}
    \caption{Measles (2012-2016), Simulation}
    \label{fig-Measles-Simul}
\end{figure}

Can be seen, in a general way,  that our model behaves in the same manner as the real data. It seems that if the vaccination policy in France, continues to be very strict, it may decrease the number of cases.
\section{conclusions} \label{section:conclusions}
We considered a discretization procedure, based on Mickens NSFD scheme, to get a discrete-time model from a continuous time with vaccination and incidence given by a general function. For a family of models containing the previous discrete-time model, we achieved results on the persistence and the extinction of the disease (Theorems~\ref{teo:Extinction} and~\ref{teo:Permanence}). They contain the results of Zhang~\cite{Zhang-AMC-2015} as a particular case. Our threshold conditions depend on the parameters of the model and of the incidence function derivative, with respect to the infectives, computed on some disease-free solution. This agrees with the continuous counterparts of these results~\cite{Pereira-Silva-Silva}.

We also considered the problem of establishing the consistency of the continuous-time model and the discrete-time model for small time-steps, in the sense that if the time step is small enough when we have persistence (respectively extinction) for the continuous-time model we also have persistence (respectively extinction) for the discrete-time model (at least for situations where Theorems~\ref{teo:Permanence-Pereira-Silva-Silva} and~~\ref{teo:Extinction-Pereira-Silva-Silva} allow us to conclude that we have persistence or extinction). Assuming the differentiability of parameters, our result on this direction, Theorem~\ref{teo:consitency}, furnishes an interval $[0, a]$, where $a$ depend only on the parameters of the model and their derivatives, where there is consistency.

We present an example of a periodic system of period $1$ where the continuous and the discrete-time system with time-step $h=1/L$ are not consistent. Namely, for that time-step,  we will have persistence for the continuous time model and extinction for the discrete-time model. These examples show the importance of knowing that for time steps smaller than some explicit value we have consistency, a type of result like the one in Theorem~\ref{teo:consitency}.

Finally, we carried out some simulations to illustrate our results. As one might expect our simulations furnish evidence that we may have consistency in intervals whose length are several times bigger than the length of the given interval in Theorem~\ref{teo:consitency}. Additionally, we used our model to describe a real situation, namely the case of measles incidence in France in the period 2012-2016, and compared our results with the real-time series for the infectives. We found in general, the predictive behaviour of our model very similar to the real data.

\appendix
\section{Proof of the results in Section~\ref{section:PE}} \label{Appendix:A}

We begin this section by noting a simple consequence of our assumptions that will be used several times throughout the proofs: it follows from H\ref{cond-P1}) that there are constants $K>0$ and $\theta \in ]0,1[$ such that
\begin{equation}\label{cond-P1-a}
\prod_{k=m}^{n-1} \dfrac{1}{1+\mu_k} < K \theta^{n-m},
\end{equation}
for $m,n \in \N$ sufficiently large. Additionally, using H\ref{cond-H1}), H\ref{cond-H2}) and H\ref{cond-P3}), we have
\[
\begin{split}
\frac{\vphi(x,y)}{y}
& = \frac{\vphi(x,y)-\vphi(x,0)}{y-0} \le \frac{\vphi(x,y)-\vphi(x,0)}{y-0} \le \partial_2\vphi(x,0) \\
& = |\partial_2\vphi(x,0)-\partial_2\vphi(0,0)| \le k_\vphi x
\end{split}
\]
and thus
\begin{equation}\label{cond-P3-a1}
\vphi(x,y) \le k_\vphi xy.
\end{equation}
Similarly
\begin{equation}\label{cond-P3-a2}
\psi(x,y) \le k_\psi xy.
\end{equation}

We will now proceed with the proof of the results in section~\ref{section:PE}.\\[.5cm]

\begin{proof}[Proof of Lemma~\ref{lema:system}]
Let $S_n>0$, $I_n>0$, $R_n>0$ and $V_n>0$. By~\eqref{eq:ProblemaPrincipal}, \eqref{cond-P3-a1} and~\eqref{cond-P3-a2}, we obtain
    \[
    V_{n+1}=\dfrac{V_n + p_n S_{n+1}}{1+\mu_n+\eta_n + \sigma_n k_{\psi}I_n}
    \]
    and thus
    \[
    S_{n+1} \ge \dfrac{\eta_n(V_n+\Lambda_n + S_n)+(\Lambda_n + S_n)(1 + \mu_n + \sigma_n k_{\psi}I_n)}{(1+\mu_n+\beta_nk_{\vphi}I_n)(1+\mu_n+\eta_n+\sigma_n k_{\vphi}I_n) + p_n(1 + \mu_n + \sigma_n k_{\psi}I_n)}.
    \]
Therefore we conclude that $S_{n+1}>0$ and $V_{n+1}>0$. By the second and third equations in~\eqref{eq:ProblemaPrincipal} we obtain
    \[
    I_{n+1} \ge \dfrac{I_n}{1+\mu_n+\alpha_n+\gamma_n}
    \]
and
    $$R_{n+1}=\dfrac{\gamma_nI_{n+1}+R_n}{1+\mu_n}$$
and we conclude that $I_{n+1}>0$ and $R_{n+1}>0$.
The previous inequalities allow us to conclude by induction that $S_n >0$, $I_n >0$, $R_n >0$ and $V_n >0$ for all $n \in \N$. In the same way we can conclude that, if $S_0\ge 0$, $I_0 \ge0$, $R_0 \ge0$ and $V_0 \ge0$, then $S_n \ge0$. $I_n \ge0$, $R_n \ge0$ and $V_n \ge0$ for all $n \in \N$. This proves~\ref{cond-1-bs}) and~\ref{cond-2-bs}) in Lemma~\ref{lema:system}.

 By~\eqref{eq:ProblemaPrincipal}, we have
$$N_{n+1} \le \dfrac{\Lambda_n}{1+\mu_n}+\dfrac{N_n}{1+\mu_n},$$
where $N_n=S_n+I_n+R_n+V_n$ is the total population. By Lemma 2 in~\cite{Zhang-Teng-Gao-AA-2008} we obtain the result.
\end{proof}
\ \\
\begin{proof}[Proof of Lemma~\ref{lema:indep}]
To show that $\mR_D^\ell(\xi^{*},\lambda)$ is independent of the selection of $\xi^{*}=(x_n^{*},y_n^{*})$, a  fixed solution of~\eqref{eq:SistemaAuxiliar}, it is important to note that according to~\ref{cond-2-aux}) in Lemma~\ref{lema:auxSystem}, for any $\varepsilon>0$ and any solution $\xi=(x_n,y_n)$ of system~\eqref{eq:SistemaAuxiliar} with initial value $x_0>0$, $y_0 >0$, there exists an $N \in \N^{+}$ such that, for $k \geq N$, we have $|x_{k}-x^{*}_{k}| \leq \varepsilon$ and $|y_{k}-y^{*}_{k}| \leq \varepsilon$. Hence
$$x^{*}_{k}- \varepsilon \leq x_{k} \leq x^{*}_{k} + \varepsilon \qquad \qquad y^{*}_{k}- \varepsilon \leq y_{k} \leq y^{*}_{k} + \varepsilon. $$

By~H\ref{cond-H1}), we have
$$|\partial_2 \vphi(x_{k},0)- \partial_2 \vphi(x^{*}_{k},0)| \leq k_{\vphi}|x_{k}-x^{*}_{k}|\leq k_{\vphi}\varepsilon$$
and
$$|\partial_2 \psi(y_{k},0)- \partial_2 \psi(y^{*}_{k},0)| \leq k_{\psi}|y_{k}-y^{*}_{k}|\leq k_{\psi}\varepsilon.$$
So,
$$\partial_2 \vphi(x^{*}_{k},0)-k_{\vphi}\varepsilon \leq \partial_2 \vphi(x_{k},0) \leq \partial_2 \vphi(x^{*}_{k},0)+ k_{\vphi}\varepsilon  $$
 and
$$\partial_2 \psi(y^{*}_{k},0)-k_{\psi}\varepsilon \leq \partial_2 \psi(y_{k},0) \leq \partial_2 \psi(y^{*}_{k},0)+ k_{\psi}\varepsilon.  $$
Combining the previous computations, we get
\begin{equation}\label{eq:enquadramento}
\begin{split}
 & \frac{1+\beta_k \partial_2 \vphi(x^{*}_{k+1},0)+\sigma_k \partial_2 \psi(y^{*}_{k+1},0)- \overline{L}_k\varepsilon}{1+\mu_k+\alpha_k+\gamma_k} \\
   & \leq \frac{1+\beta_k \partial_2 \vphi(x_{k+1},0)+\sigma_k \partial_2 \psi(y_{k+1},0)}{1+\mu_k+\alpha_k+\gamma_k} \\
   & \leq \frac{1+\beta_k \partial_2 \vphi(x^{*}_{k+1},0)+\sigma_k \partial_2 \psi(y^{*}_{k+1},0) + \overline{L}_k\varepsilon}{1+\mu_k+\alpha_k+\gamma_k},
\end{split}
\end{equation}
where $\overline{L}_k=\beta_k k_{\vphi} + \sigma_k k_{\psi}.$

Let $$ r_k = \frac{1+\beta_k \partial_2 \vphi(x^{*}_{k+1},0)+\sigma_k \partial_2 \psi(y^{*}_{k+1},0)}{1+\mu_k+\alpha_k+\gamma_k}.$$
Using~H\ref{cond-C1}) and~\ref{cond-4-aux}) in Lemma~\ref{lema:auxSystem}, it's easy to see that
$$r_k \leq  \frac{1+2\beta^u k_{\vphi}M+2\sigma^u k_{\psi}M }{1+\mu^l+\alpha^l+\gamma^l} =:r,$$
for sufficiently large $k \in \N$, and that
$$\overline{L}=\frac{\overline{L}_k}{1+\mu_k+\alpha_k+\gamma_k}\leq \frac{\beta^u k_{\vphi} + \sigma^u k_{\psi}}{1+\mu^l+\alpha^l+\gamma^l} =:C.$$
So, for sufficiently large $n$,
\[
\prod_{k=n}^{n+\lambda}\left(r_k+\frac{\overline{L}_k\varepsilon}{1+\mu_k+\alpha_k+\gamma_k}\right) \leq  \prod_{k=n}^{n+\lambda}(r_k+C\varepsilon) = \prod_{k=n}^{n+\lambda} r_k + \Theta_\eps
\]
where
\[
\Theta_\eps = \binom{\lambda +1}{\lambda} r^{\lambda}C\varepsilon+ \ldots + \binom{\lambda +1}{1} rC^{\lambda}\varepsilon^{\lambda}+ C^{\lambda+1}\varepsilon^{\lambda+1}.
\]
Analogously
\[
\prod_{k=n}^{n+\lambda}\left(r_k-\frac{\overline{L}_k\varepsilon}{1+\mu_k+\alpha_k+\gamma_k}\right) \geq \prod_{k=n}^{n+\lambda} r_k - \Theta_\eps.
\]

By~\eqref{eq:enquadramento}, we obtain
$$-\Theta_\eps+\underset{n \to +\infty}{\liminf} \prod_{k=n}^{n+\lambda}r_k \leq \mR_D^\ell(\xi^{*},\lambda) \leq \Theta_\eps+\underset{n \to +\infty}{\liminf} \prod_{k=n}^{n+\lambda}r_k.$$
Thus
$$\left|\mR_D^\ell(\xi^{*},\lambda)- \mR_D^\ell(\xi,\lambda)\right| < \Theta_\eps$$
and, by the arbitrariness of $\varepsilon$, we obtain $\mR_D^\ell(\xi,\lambda)= \mR_D^\ell(\xi^{*},\lambda)$.
Replacing $\liminf$ by $\limsup$ in the preceding argument, we reach a similar conclusion for $\mR_D^u(\xi^{*},\lambda)$.
The result follows.
\end{proof}
\ \\
\begin{proof}[Proof of Theorem~\ref{teo:Extinction}]
First note that the original system~\eqref{eq:ProblemaPrincipal} can be rewritten as follows:
\small{
\begin{equation}\label{eq:ProblemaPrincipal1}
\begin{cases}
S_{n+1}= \frac{1}{1+\mu_n+p_n}( \Lambda_n + S_n - \beta_n\vphi(S_{n+1},I_n)+\eta_nV_{n+1})\\ I_{n+1}=\frac{1}{1+\mu_n+\alpha_n+\gamma_n}(\beta_n\vphi(S_{n+1},I_n)+\sigma_n\psi(V_{n+1},I_n)+I_n) \\
R_{n+1}=\frac{1}{1+\mu_n}(\gamma_nI_{n+1} + R_n)\\
V_{n+1}=\frac{1}{1+\mu_n+\eta_n}(p_nS_{n+1}- \sigma_n\psi(V_{n+1}, I_n) + V_n)
\end{cases},
\end{equation}
}
$n=0,1,\ldots$.

Firstly, we will establish~\ref{teo:Extinction-a}).
Since $\mR_D^u(\lambda)<1$, we can choose $\varepsilon_0>0$, $\varepsilon \in ]0,1[$ and a  sufficiently large integer $N_1 \in \N$ such that
\begin{equation}\label{extintion:1}
\prod_{k=n}^{n+\lambda}\frac{1+\beta_k \partial_2 \vphi(x_{k+1},0)+\sigma_k \partial_2 \psi(y_{k+1},0)+(\beta^u k_\vphi+\sigma^u k_\psi)\varepsilon_0}{1+\mu_k+\alpha_k+\gamma_k}< \varepsilon
\end{equation}
for all $n \geq N_1.$

For any solution $(S_n,I_n,R_n,V_n)$ of~\eqref{eq:ProblemaPrincipal1} with initial conditions $S_0 >0$, $I_0>0$, $R_0>0$ and $V_0>0$, we have

\begin{equation}
\begin{cases}
S_{n+1} \leq \dfrac{\Lambda_n + \eta_n V_{n+1} +S_n}{1+\mu_n+p_n} \\[2mm]
V_{n+1} \leq \dfrac{p_n S_{n+1} +V_n}{1+\mu_n+\eta_n}
\end{cases}.
\end{equation}

By the comparison principle, we obtain $S_n \leq x_n$ and $V_n \leq y_n$ for all $n \in \mathbb{N}$, where $(x_n, y_n)$ is the solution of ~\eqref{eq:SistemaAuxiliar} with initial condition $(x_0,y_0)=(S_0,V_0)$. According to Lemma~\ref{lema:auxSystem}, the solution $(x_n^*,y_n^*)$ is globally uniformly attractive and thus, for the aforementioned $\varepsilon_0>0$, there exists an $N_2 \in \mathbb{N}$ such that
$$|x_n-x_n^*|\leq \varepsilon_0 \, \, \textrm{and} \,\,|y_n-y_n^*| \leq \varepsilon_0 \,\textrm{for all}\, n \geq N_2$$
From this, it may be concluded that
\begin{equation}\label{eq:s-maj-x*+eps0}
S_n \leq x_n^*+\varepsilon_0\,\, \textrm{and}\,\, V_n\leq y_n^*+ \varepsilon_0 \,\, \textrm{for all}\,\, n \geq N_2.
\end{equation}
By the second equation of ~\eqref{eq:ProblemaPrincipal} we get
\begin{equation}\label{eq:majora-ext-1}
\begin{split}
I_{n+1}
& = \dfrac{1}{1+\mu_n+\alpha_n+\gamma_n}\big( \beta_n \vphi(S_{n+1},I_n)+\sigma_n \psi(V_{n+1},I_n)+I_n\big) \\
& = \dfrac{1}{1+\mu_n+\alpha_n+\gamma_n}\Big(\beta_n \dfrac{ \vphi(S_{n+1},I_n)}{I_n}+\sigma_n \dfrac{\psi(V_{n+1},I_n)}{I_n}+1\Big)I_n.
\end{split}
\end{equation}
By~H\ref{cond-P3}), we have
\begin{equation}\label{eq:majora-ext-2}
\dfrac{\vphi(S_{n+1},I_n)}{I_n} \leq \partial_2 \vphi(S_{n+1},0) \,\,\textrm{ and }\,\,\dfrac{\psi(V_{n+1},I_n)}{I_n} \leq \partial_2 \psi(V_{n+1},0).
\end{equation}

By ~H\ref{cond-H1}), $x \mapsto \partial_2 \vphi(x,0)$ and $x \mapsto \partial_2 \psi(x,0)$ are non decreasing and also Lipschitz, so, using~\eqref{eq:s-maj-x*+eps0} we obtain
\[
\begin{split}
 \partial_2\vphi(S_{n+1},0)-\partial_2\vphi(x_{n+1}^*,0)
& \leq \partial_2\vphi(x^*_{n+1}+\eps_0,0)-\partial_2\vphi(x_{n+1}^*,0) \\
& = |\partial_2\vphi(x^*_{n+1}+\eps_0,0)-\partial_2\vphi(x_{n+1}^*,0)| \\
& \leq k_{\vphi} \varepsilon_0
\end{split}
\]
and thus
\begin{equation}\label{eq:AAA}
\partial_2\vphi(S_{n+1},0) \leq \partial_2\vphi(x_{n+1}^*,0) + k_{\vphi} \varepsilon_0.
\end{equation}
Analogously
\begin{equation}\label{eq:BBB}
\partial_2\psi(V_{n+1},0) \leq \partial_2 \psi(y_{n+1}^*,0) + k_{\psi} \varepsilon_0.
\end{equation}

Therefore, by~\eqref{extintion:1},~\eqref{eq:majora-ext-1},~\eqref{eq:majora-ext-2},~\eqref{eq:AAA} and~\eqref{eq:BBB} we have
$$
I_{n+1} \le \dfrac{\beta_n ( \partial_2\vphi(x_{n+1}^*,0) + k_{\vphi} \varepsilon_0)+\sigma_n (\partial_2 \psi(y_{n+1}^*,0) + k_{\psi} \varepsilon_0)+1}{1+\mu_n+\alpha_n+\gamma_n} I_n
\le \eps I_n,
$$
for all $n \geq N_2$. We conclude that $I_m \le \eps^{m-N_2}I_{N_2} \rightarrow 0$ as $m\rightarrow \infty$. This completes the proof of~\ref{teo:Extinction-a}).

Next, to establish~\ref{teo:Extinction-b}), let us consider two arbitrary solutions of the original system $(S_n^{(1)},I_n^{(1)} ,V_n^{(1)} ,R_n^{(1)} )$ and $(S_n^{(2)},I_n^{(2)} ,V_n^{(2)} ,R_n^{(2)} )$  and $\lambda$ a constant such that $R_0^u(\lambda) <1$.
Let $\iota_n=I_n^{(1)}-I_n^{(2)}$ and $\rho_n=R_n^{(1)}-R_n^{(2)}$. By~\eqref{eq:ProblemaPrincipal-disc}, we have
\begin{eqnarray*}
\rho_{n+1}-\rho_{n}
& = & (R_{n+1}^{(1)}-R_n^{(1)})-(R_{n+1}^{(2)}-R_n^{(2)}) \\
& = & \gamma_n (I_{n+1}^{(1)}-I_{n+1}^{(2)}) - \mu_n (R_{n+1}^{(1)}-R_{n+1}^{(2)})\\
& = & \gamma_n \iota_{n+1} - \mu_n \rho_{n+1}.
 \end{eqnarray*}
Because $R_0^u(\lambda) < 1$, we conclude that $\iota_n \to 0$ as $n \to +\infty$ and therefore, given $\eps>0$, there is $N \in \N$ sufficiently large such that, for $n \ge N$,
$$\rho_{n+1} + \mu_n \rho_{n+1} = \gamma_n\iota_{n+1} +\rho_{n} <  \varepsilon + \rho_{n}$$
Thus, since $\mu_n>0$, we get
$$\rho_{n+1} < \frac{\varepsilon}{1+ \mu_n} + \frac{\rho_{n}}{1+ \mu_n}.$$
and proceeding by induction
\begin{equation}\label{eq:rho2}
\rho_{n+1}
<\left(\prod_{m=0}^{n-1}\dfrac{1}{1+\mu_m}\right)\rho_0 + \eps \sum_{m=0}^{n-1} \left(\prod_{k=m}^{n-1}\dfrac{1}{1+\mu_k}\right).
\end{equation}
By~H\ref{cond-P1}) and~\eqref{cond-P1-a}, we conclude that
$$\limsup_{n \to +\infty} \rho_n = 0.$$
Thus $\rho_n = R_n^{(1)}-R_n^{(2)} \to 0$ as $n \to +\infty$.
Similar computations show that $S_n^{(1)}-S_n^{(2)} \to 0$ and $V_n^{(1)}-V_n^{(2)} \to 0$. This proves~\ref{teo:Extinction-b}) and the result follows.
\end{proof}
\ \\
\begin{proof}[Proof of Theorem~\ref{teo:Permanence}]
Since $\mathcal R_0^\ell(\lambda)>1$, there are $\eps,\eps_0>0$ such that
\begin{equation}\label{eq:cond-perm-eps0}
\liminf_{n \to +\infty} \prod_{k=n}^{n+\lambda} \frac{1+\beta_k \partial_2 \vphi(x_{k+1}^*,0)+\sigma_k \partial_2 \psi(y_{k+1}^*,0)-u}{1+\mu_k+\alpha_k+\gamma_k}>1+\eps,
\end{equation}
for all $u \in [0,\eps_0]$.

We claim that there is $\eps_1>0$ such that
\begin{equation}\label{eq:claim-persist}
\limsup_{n\to\infty}I_n>\eps_1
\end{equation}
for every solution with positive initial conditions of system~\eqref{eq:ProblemaPrincipal}.

We proceed by contradiction. Assume that~\eqref{eq:claim-persist} doesn't hold. Then, for each $\eps_1$ there is $N_1 \in \N$ and a solution $(S_n,I_n,R_n,V_n)$ with positive initial conditions such that $I_n\le \eps_1$ for all $n \ge N_1$. By~\ref{cond-3-bs}) in Lemma~\ref{lema:system}, we can assume that $S_n,V_n<M$ for all $n \ge N_1$.
By~\eqref{cond-P3-a1} we have
\begin{equation}\label{eq:maj-phi-s-v}
\vphi(S_{n+1},I_n) = k_\vphi S_{n+1}I_n \le k_\vphi M \eps_1
\end{equation}
and likewise, by~\eqref{cond-P3-a2}, we get
\begin{equation}\label{eq:maj-psi-s-v}
\psi(V_{n+1},I_n) = k_\psi V_{n+1}I_n \le k_\psi M \eps_1.
\end{equation}

By~\eqref{eq:ProblemaPrincipal},~\eqref{eq:maj-phi-s-v} and~\eqref{eq:maj-psi-s-v} we have
\begin{equation}\label{eq:SistemaAuxiliar-proof-perm-ineq}
\begin{cases}
S_{n+1}\ge \dfrac{\Lambda_n + \eta_n V_{n+1} +S_n-\beta^u k_\vphi M\eps_1}{1+\mu_n+p_n} \\[3mm]
V_{n+1}\ge \dfrac{p_n S_{n+1} +V_n-\sigma^u k_\psi M\eps_1}{1+\mu_n+\eta_n}
\end{cases}.
\end{equation}
for all $n \ge N_1$.

Given $\eps_1>0$, consider the auxiliary system
\begin{equation}\label{eq:SistemaAuxiliar-proof-perm}
\begin{cases}
x_{n+1}=\dfrac{\Lambda_n + \eta_n y_{n+1} +x_n-\beta^u k_\vphi M\eps_1}{1+\mu_n+p_n} \\[3mm]
y_{n+1}=\dfrac{p_n x_{n+1} +y_n-\sigma^u k_\psi M\eps_1}{1+\mu_n+\eta_n}
\end{cases}.
\end{equation}
For any $n_0 \in \N$ and $x_0,y_0 \in \R^+$, let $(x_n,y_n)$ be the solution of~\eqref{eq:SistemaAuxiliar} with initial condition $(x_{n_0},y_{n_0})=(x_0,y_0)$ and let $(\bar x_{n_0},\bar y_{n_0})=(x_0,y_0)$ be the solution of~\eqref{eq:SistemaAuxiliar-proof-perm} with the same initial condition.
By~\ref{cond-3-aux}) in Lemma~\ref{lema:auxSystem} we obtain
\[
\sup_{n \in \N_0} \left\{ |\bar x_n - x_n| + |\bar y_n -y_n|\right\}
\le LM(\beta^u k_\vphi+\sigma^u k_\psi)\eps_1
\]
and thus we can take $\eps_1>0$ small enough such that
\[
\sup_{n \in \N_0} \left\{ |\bar x_n - x_n| + |\bar y_n -y_n|\right\} \le \frac{\eps_0}{2}.
\]
On the other hand, by~\ref{cond-2-aux}) in Lemma~\ref{lema:auxSystem}, there is $N_2\ge N_1$ sufficiently large such that
\[
|x_n^* - x_n| + |y_n^* -y_n| \le \frac{\eps_0}{2},
\]
for all $n \ge N_2$. Therefore
\begin{equation}\label{eq:xn-xnstar+yn-ynstar}
|\bar x_n-x_n^*| + |\bar y_n-y_n^*| \le \eps_0,
\end{equation}
for all $n \ge N_2$.

Noting that~\eqref{eq:SistemaAuxiliar-proof-perm-ineq} can be written as
$$
\begin{cases}
\small S_{n+1}\ge \dfrac{\eta_n p_n(~1+\mu_n)(1+\mu_n+p_n+\eta_n)(\Lambda_n+S_n-\beta^u Mk_\vphi \eps_1)}{(1+\mu_n)(1+\mu_n+p_n)(1+\mu_n+p_n+\eta_n)}\\[3mm]
\quad\quad\quad\quad\quad\quad\quad\quad+\dfrac{(V_n-\sigma^u Mk_\psi\eps_1)(1+\mu_n+p_n)}{(1+\mu_n)(1+\mu_n+p_n)(1+\mu_n+p_n+\eta_n)}\\[3mm]
\small V_{n+1}\ge \dfrac{p_n(\Lambda_n+S_n-\beta^u Mk_\vphi \eps_1)+(V_n-\sigma^u Mk_\psi\eps_1)(1+\mu_n+p_n)}{(1+\mu_n)(1+\mu_n+p_n+\eta_n)}
\end{cases}
$$
and, using~\eqref{eq:SistemaAuxiliar-proof-perm} and~\eqref{eq:xn-xnstar+yn-ynstar}, we conclude that
\[
S_n \ge \bar x_n \ge x_n^* -\eps_0 \quad \text{and} \quad V_n \ge \bar y_n \ge y_n^* -\eps_0,
\]
for all $n \ge N_2$. Thus, since $S_n \le x_n^*$ and $V_n \le y_n^*$ for all $n \in \N$, we have
\begin{equation}\label{eq:maj-x-y-by-s-v}
|S_n - x_n^*| \le \eps_0 \quad \text{and} \quad |V_n - y_n^*| \le \eps_0,
\end{equation}
for all $n \ge N_2$. By~H\ref{cond-H1}), we have, for all $n \ge N_2$,
\[
|\partial_2 \vphi(S_{k+1},0)-\partial_2 \vphi(x^*_{k+1},0)| \le k_\vphi |S_{k+1}-x^*_{k+1}| \le k_\vphi  \eps_0
\]
and thus
\begin{equation}\label{eq:maj-phixn+1}
\partial_2 \vphi(x^*_{k+1},0) - k_\vphi \eps_0 \le \partial_2 \vphi(S_{k+1},0) \le \partial_2 \vphi(x^*_{k+1},0)+k_\vphi \eps_0.
\end{equation}
Reasoning similarly we obtain, for all $n \ge N_2$,
\begin{equation}\label{eq:maj-psiyn+1}
\partial_2 \psi(y^*_{k+1},0) - k_\psi \eps_0 \le \partial_2 \psi(V_{k+1},0) \le \partial_2 \psi(y^*_{k+1},0)+k_\psi \eps_0.
\end{equation}
By~H\ref{cond-H1}) and H\ref{cond-H2}), we conclude that
\[
\begin{split}
\vphi(S_{n+1},I_n)
& =\vphi(S_{n+1},0)+\partial_2\vphi(S_{n+1},\eps_n)(I_n-0)
= \partial_2\vphi(S_{n+1},\eps_n)I_n,
\end{split}
\]
for some $\eps_n \in [0,\eps_1]$, and all $n \ge N_2$. Thus, by continuity of $\partial_2 \vphi$,
\[
\dfrac{\vphi(S_{n+1},I_n)}{I_n}=\partial_2\vphi(S_{n+1},\eps_n) \ge \partial_2\vphi(S_{n+1},0)-\theta_1(\eps_1),
\]
with $\theta_1(\eps_1)\to 0$ as $\eps_1 \to 0$, for all $n \ge N_2$. Thus, by~\eqref{eq:maj-phixn+1}
\[
\dfrac{\vphi(S_{n+1},I_n)}{I_n}\ge \partial_2 \vphi(x^*_{k+1},0)-k_\vphi \eps_0-\theta_1(\eps_1),
\]
where $\theta_1(\eps_1)\to 0$ as $\eps_1 \to 0$. Similarly, for all $n \ge N_2$, we have, by~\eqref{eq:maj-psiyn+1}
\[
\dfrac{\psi(V_{n+1},I_n)}{I_n} \ge \partial_2 \psi(y^*_{k+1},0)-k_\psi \eps_0-\theta_2(\eps_1),
\]
where $\theta_2(\eps_1) \to 0$ as $\eps_1 \to 0$. From the second equation in~\eqref{eq:ProblemaPrincipal}, we have
\begin{equation}\label{eq:cardi}
\begin{split}
I_{n+1}
& = \dfrac{1+\beta_n\vphi(S_{n+1},I_n)/I_n+\sigma_n\psi(V_{n+1},I_n)/I_n}{1+\mu_n+\alpha_n+\gamma_n} I_n\\
& \ge \dfrac{1+\beta_n\partial_2 \vphi(x^*_{k+1},0) +\sigma_n\partial_2 \psi(y^*_{k+1},0)-u}{1+\mu_n+\alpha_n+\gamma_n} I_n.
\end{split}
\end{equation}
where
\begin{equation}\label{eq:cardi2}
u=\beta^u(k_\vphi \eps_0+\theta_1(\eps_1)) +\sigma^u(k_\psi \eps_0+\theta_2(\eps_1)).
\end{equation}
for all $n \ge N_2$. Letting $\eps_1$ be sufficiently small, we conclude, according to~\eqref{eq:cond-perm-eps0} and \eqref{eq:cardi}, that $I_n\to +\infty$ as $n \to +\infty$. A contradiction. Thus, we conclude that~\eqref{eq:claim-persist} holds.

Next we will prove the permanence of the infectives. By~\ref{cond-3-bs}) in Lemma~\ref{lema:system}, it is only necessary to prove that there is an $\eps_2>0$ such
that, for any solution $(S_n, I_n, R_n, V_n)$ of~\eqref{eq:ProblemaPrincipal} with positive initial conditions, we have
\begin{equation}\label{eq:caim-strong-persist-2}
\liminf_{n \to +\infty} I_n > \eps_2.
\end{equation}

Recall that, since $\mathcal R_0^\ell(\lambda)>1$, there are $\eps,\eps_0>0$ such that~\eqref{eq:cond-perm-eps0} holds for all $u \in [0,\eps_0]$.

If~\eqref{eq:caim-strong-persist-2} doesn't hold, then, given $\eps_0>0,$ there must be a sequence of solutions of~\eqref{eq:ProblemaPrincipal}, $((S_{n,k},I_{n,k},R_{n,k},V_{n,k})_{n\in\N})_{k \in \N}$,
with initial conditions $(S_{0,k},I_{0,k},R_{0,k},V_{0,k})$ such that
$$ \liminf_{n \to +\infty} I_{n,k} < \frac{\eps_0}{k^2}.$$

From~\eqref{eq:claim-persist}, for each $k \in \N$, there must be two sequences $(s_{m,k})_{m \in \N}$ and $(t_{m,k})_{m \in \N}$ such that
$s_{m,k} \to +\infty$ as $m \to +\infty$,
\[
0<s_{1,k}<t_{1,k}<s_{2,k}<t_{2,k}<\cdots <s_{m,k}<t_{m,k}<\ldots,
\]
\begin{equation}\label{eq:B}
I_{s_{m,k},k}>\eps_0/k, \quad I_{t_{m,k},k}<\eps_0/k^2
\end{equation}
and
\[
\frac{\eps_0}{k^2} \le I_{n,k} \le \frac{\eps_0}{k}, \ \text{for all} \ n \in [s_{m,k}+1,t_{n,k}-1] \cup \N.
\]

Given $n \in [s_{m,k},t_{m,k}-1] \cup \N$, we have
\[
\begin{split}
I_{n+1,k}
& = \dfrac{1+\beta_n\vphi(S_{n+1,k},I_{n,k})/I_{n,k}
+\sigma_n\psi(V_{n+1,k},I_{n,k})/I_{n,k}}{1+\mu_n+\alpha_n+\gamma_n} I_{n,k}\\
& \ge \dfrac{1}{1+\mu^u+\alpha^u+\gamma^u} I_{n,k}
\end{split}
\]
and therefore
\begin{equation}\label{eq:C}
\eps_0/k^2 > I_{t_{n,k},k} \ge \sigma^{t_{n,k}-s_{n,k}} I_{s_{n,k},k} > \sigma^{t_{n,k}-s_{n,k}}\eps_0/k,
\end{equation}
where
\[
\sigma = \dfrac{1}{1+\mu^u+\alpha^u+\gamma^u}.
\]
Thus, by~\eqref{eq:B} and~\eqref{eq:C},
\begin{equation}\label{eq:afastamento}
t_{n,k}-s_{n,k} \ge \dfrac{\ln k}{\ln (1/\sigma)} \to +\infty \ \text{as} \ k \to +\infty.
\end{equation}
In view of~\eqref{eq:afastamento}, we can choose $k_0 \in \N$ such that
\[
t_{n,k}-s_{n,k} > N+\lambda+1,
\]
for all $k \ge k_0$.

Letting $m$ and $k \ge k_0$ be sufficiently large and $\eps_2>0$ be sufficiently small, we may assume that
\begin{equation}\label{eq:SistemaAuxiliar-proof-perm-ineq2}
\begin{cases}
S_{n,k}\ge \dfrac{\Lambda_n + \eta_n V_{n,k} +S_{n,k}-\beta^uk_\vphi M\eps_2}{1+\mu_n+p_n} \\[3mm]
V_{n,k}\ge \dfrac{p_n S_{n,k} +V_{n,k}-\sigma^uk_\psi M\eps_2}{1+\mu_n+\eta_n}
\end{cases}.
\end{equation}
holds for all $n \in [s_{m,k}+1,t_{m,k}-1] \cap \N$.

Let $(\bar x_n, \bar y_n)$ be a solution of~\eqref{eq:SistemaAuxiliar-proof-perm} with initial value $(\bar x_{s_{m,k}},\bar y_{s_{m,k}})=(S_{s_{m,k}},V_{s_{m,k}})$. We have $S_n\ge \bar x_n$ and $V_n\ge \bar y_n$ for all
$n \in [s_{m,k},t_{m,k}] \cap \N$. Letting $\eps_2>0$ in~\eqref{eq:SistemaAuxiliar-proof-perm-ineq2} be sufficiently small and $(x_n,y_n)$ be the solution of~\eqref{eq:SistemaAuxiliar} with $x_{s_{n,k}+1}=S_{s_{n,k}+1}$ and $y_{s_{n,k}}=V_{s_{n,k}}$, we have by~\ref{cond-3-aux}) in Lemma~\ref{lema:auxSystem}
\[
|x_n-\bar x_n|+|y_n-\bar y_n| \le \frac{\eps_0}{2}
\]
for all $n \in [s_{m,k},t_{m,k}] \cap \N$. we conclude that
$$S_n\ge \bar x_n\ge x_n -\eps_0/2 > x_n^*-\eps_0$$
and
$$V_n\ge \bar y_n\ge y_n -\eps_0/2 > y_n^*-\eps_0$$
for all $n \in [s_{m,k},t_{m,k}] \cap \N$.

Proceeding like before, we obtain~\eqref{eq:maj-phixn+1} and~\eqref{eq:maj-psiyn+1} with $S_k$ and $V_k$ replaced by $S_{n+1,k}$ and
$V_{n+1,k}$ respectively and, for sufficiently large $n \in \N$,
\[
\dfrac{\vphi(S_{n+1,k},I_{n,k})}{I_{n,k}}\ge \partial_2 \vphi(x^*_{k+1},0)-k_\vphi \eps_0-\theta_1(\eps_1),
\]
and
\[
\dfrac{\psi(V_{n+1,k},I_{n,k})}{I_{n,k}} \ge \partial_2 \psi(y^*_{k+1},0)-k_\psi \eps_0-\theta_2(\eps_1),
\]
where $\theta_1(\eps_1)\to 0$ as $\eps_1 \to 0$. Therefore
\[
I_{n+1,k} \ge \dfrac{1+\beta_n \partial_2\varphi(x_{n+1}^*,0)+\sigma_n \partial_2\psi(y_{n+1}^*,0)-u}{1+\mu_n+\alpha_n+\gamma_n}I_{n,k},
\]
for all $n \in [s_{m,k},t_{m,k}] \cap \N$, where $u$ is given by~\eqref{eq:cardi2}. Thus
\[
\begin{split}
\dfrac{\eps_0}{k^2}
& > I_{t_{m,k},k} \\
& \ge I_{t_{m,k}-\lambda,k} \prod_{n=t_{m,k}-\lambda}^{t_{m,k}}
\frac{1+\beta_n \partial_2 \vphi(x_{n+1}^*,0)+\sigma_n \partial_n \psi(y_{n+1}^*,0)-u}{1+\mu_n+\alpha_n+\gamma_n}
> \dfrac{\eps_0}{k^2}
\end{split}
\]
which is a contradiction. The theorem follows.
\end{proof}

\bibliographystyle{elsart-num-sort}

\begin{thebibliography}{10}
\expandafter\ifx\csname url\endcsname\relax
  \def\url#1{\texttt{#1}}\fi
\expandafter\ifx\csname urlprefix\endcsname\relax\def\urlprefix{URL }\fi

\bibitem{Allen-Jones-Martin-MB-1994} L. J. S. Allen, M. A. Jones, C. F. Martin, A discrete-time model with vaccination for a measles epidemic, Math. Biosci., 105 (1991), 111-131

\bibitem{Allen-MB-1975} L. J. S. Allen, Some discrete-time si,sir and sis epidemic models, Math. Biosci., 124 (1994), 83-105

\bibitem{Bacaer-2014} Baca$\ddot{e}$r, Ait Dads el H., On the probability of extinction in a periodic environment, J.Math. Biol., (2014), 68:533-548

\bibitem{Cooke-TPB-1975} K. L. Cooke, A discrete-time epidemic model with classes of infectives and susceptibles, Theor. Popul. Biol., 7(1975), 175-196

\bibitem{Cui-Yang-Zhang-JDEA-2014}  Q. Cui, X. Yang and Q. Zhang, An NSFD scheme for a class of SIR epidemic models with vaccination and treatment, J. Differ. Equ. Appl. 20 (2014), 416--422

\bibitem{Cui-Zhang-JDEA-2015} Qianqian Cui and Qiang Zhang, Global stability of a discrete SIR epidemic model with vaccination and treatment, Journal of Difference Equations and Applications, 2 (2015), 111-117

\bibitem{DeJong-Diekmann-Heesterbeek-MB-1994} M. C. M. DeJong, O. Diekmann and J. A. P. Heesterbeek, The computation of r0 for discrete-time epidemic models with dynamic heterogeneity, Math. Biosci., 119 (1994), 97-114

\bibitem{Diekmann-Heesterbeek-W-2000} O. Diekmann and J. A. P. Heesterbeek, (2000), Mathematical Epidemiology of Infectious Diseases: Model Building, Analysis and Interpretation,
    Wiley, Chichester

\bibitem{Ding-Ding-JDEA-2013} D. Ding and X. Ding, A non-standard finite difference scheme for an epidemic model with
vaccination, J. Differ. Equ. Appl. 19 (2013), 179--190

\bibitem{Ding-Ding-JDEA-2014} D. Ding and X. Ding, Dynamic consistent non-standard numerical scheme for a dengue disease transmission model, J. Differ. Equ. Appl. 20 (2014), 492--505

\bibitem{Hattaf-Lashari-Boukari-Yousfi-DEDS-2015} Khalid Hattaf, Abid Ali Lashari, Brahim El Boukari and Noura Yousfi, Effect of Discretization on Dynamical Behavior in an Epidemiological Model, Differ Equ Dyn Syst, 23 (2015), 403-413

\bibitem{Hua-Teng-Zhang-MCS-2014} Hua, Teng and Zhang, Threshold conditions for a discrete nonautonomous SIRS model, Mathematics and Computers in Simulation, 97 (2014), 80-93

\bibitem{Hu-Teng-Jia-Zhang-Zhang-ADE-2014} Hu, Teng, Jia, Zhang and Zhang, Dynamical analysis and chaos control of a discrete SIS epidemic model, Advances in Difference Equations 2014, 2014:58

\bibitem{Jang-Elaydi-CAMQ-2003} S. Jang and N. Elaydi, Difference equations from discretization of a continuous epidemic model
with immigration of infectives, Can. Appl. Math. Q. 11 (2003), 93--105

\bibitem{Kaitala-Heino-Getz-BMB-1997} V. Kaitala, M. Heino and W. M. Getz, Host-parasite dynamics and the evolution of host immunity and parasite fecundity strategies, Bull. Math. Biol., 59 (1997), 427-450

\bibitem{Khasnis-Nettleman-AMR-2005} Atul A. Khasnis and Mary D. Nettleman, Global Warming and Infectious Disease, Archives of Medical Research, 36 (2005), 689-696

\bibitem{Kloeden-Potzsche-Springer-2013} Peter E. Kloeden and Christia P{\"o}tzsche, Nonautonomous dynamical systems in the life sciences, Lecture Notes in Mathematics 2102 (2013)

\bibitem{Lena-Serio-MB-1982} G. Lena and G. Serio, A discrete method for the identification of parameters of a deterministic epidemic model, Math. Biosci., 60 (1982), 161-175

\bibitem{Lefevre-Malice-MM-1986} C. Lefevre and M. Malice, A discrete time model for a s-i-s infectious disease with a random number of contacts between individuals, Math.Model., 7 (1986), 785-792

\bibitem{Lefevre-Picard-MB-1989} C. Lefevre and P. Picard, On the formulation of discrete-time epidemic models, Math.Biosci., 95 (1989), 27-35

\bibitem{Liu-Peng-Zhang-AML-2015} Liu, Peng and Zhang, Effect of discretization on dynamical behavior of SEIR and SIR models with nonlinear incidence, Applied Mathematics Letters, 39 (2015), 60-66

\bibitem{Longini-MB-1986} I. M. Longini, The generalized discrete-time epidemic model with immunity: a synthesis, Math. Biosci., 82 (1986), 19-41

\bibitem{Mickens-JDEA-2002} Mickens, Nonstandard finite difference schemes for differential equations, J. Difference Equ. Appl. 8(2002), 823-847

\bibitem{Mickens-B-1982} R. E. Mickens, A discrete-time model for the spread of periodic diseases without immunity, Biosystems, 26 (1992), 193-198

\bibitem{Mickens-JCAM-1999} R. E. Mickens, Discretizations of nonlinear differential equations using explicit nonstandard methods, J.Comput.Appl.Math., 110 (1999), 181-185

\bibitem{Mickens-Washington-JDEA-2012} Ronald E. Mickens and Talitha M. Washington, A note on an NSFD scheme for a mathematical model of respiratory virus transmission, Journal of Difference Equations and Applications, 18 (2012), 525-529

\bibitem{Pereira-Silva-Silva} E. Pereira, C. M. Silva and J. A. L. Silva, A Generalized Non-Autonomous SIRVS Model, Math. Methods Appl. Sci., 36 (2013), 275-289

\bibitem{Spicer-BMB-1979} C. C. Spicer, The mathematical modeling of influenza, Brit. Med. Bull., 35 (1979), 23-28

\bibitem{Tan-Gao-Fan-DDNS-2015} Tan, Gao and Fan, Bifurcation Analysis and Chaos Control in a Discrete
    Epidemic System, Discrete Dynamics in Nature and Society, Volume 2015, Article ID 974868

\bibitem{Viaud-MB-1993} D. P. L. Viaud, Large deviations for discrete-time epidemic models, Math. Biosci., 117 (1993), 197-210

\bibitem{Wang-Teng-Rehim-DDNS-2014} Wang, Teng and Rehim, Lyapunov Functions for a Class of Discrete SIRS Epidemic Models with Nonlinear Incidence Rate and Varying Population Sizes, Discrete Dynamics in Nature and Society, Volume 2014, Article ID 472746

\bibitem{Wei-Le-DDNS-2015} Wei and Le, Existence and Convergence of the Positive Solutions of a Discrete Epidemic Model, Discrete Dynamics in Nature and Society, Volume 2015, Article ID 434537

\bibitem{West-Thompson-MB-1997} R. W. West and J. R. Thompson, Models for the simple epidemic, Math. Biosci., 141 (1997), 29-39

\bibitem{Zhang-Teng-Gao-AA-2008}  T. Zhang, Z. Teng and S. Gao, (2008), Threshold conditions for a nonautonomous epidemic model with vaccination, Applicable Analysis, 87, 181-199.

\bibitem{Zhang-AMC-2015} T. Zhang, Permanence and extinction in a nonautonomous discrete SIRVS epidemic model with vaccination, Appl. Math. Comput., 271 (2015), 716-729

\bibitem{Zhang-Teng-Gao-AA-2008} T. Zhang, Z. Teng and S. Gao, Threshold conditions for a non-autonomous epidemic model with vaccination, Appl. Anal., 87 (2008), 181-199

\end{thebibliography}


\end{document}